\documentclass{article}
\usepackage{fullpage}
\usepackage{times}
 \usepackage{algorithm}
 \usepackage{algorithmicx}
 \usepackage{amssymb,amsmath,euscript,mathrsfs}
 \usepackage{algpseudocode}
 \usepackage{subfigure}
 \usepackage{verbatim} 
 \usepackage{hyperref,url}  
 \usepackage{amsfonts}       
\usepackage{nicefrac}       
\usepackage{microtype}  
 \usepackage{multirow}
 \usepackage{soul} 
 \usepackage[table]{xcolor}

\newcommand{\nz}{\mathrm{nz}}
\newcommand{\RR}{\mathbb{R}}
\newcommand{\CH}{\mathcal{H}}

\newcommand{\eps}{\epsilon}
\newcommand{\purity}{\mathrm{purity}}
\newcommand{\entropy}{\mathrm{entropy}}

\usepackage{amsmath,amssymb}
\usepackage{graphicx,subfigure}
\usepackage[all]{xy}
\usepackage{url}

\usepackage{amsmath,amssymb,amscd,amsfonts}
\usepackage[normalem]{ulem}
\usepackage{graphicx}
\usepackage{color} 
\graphicspath{{figs/}}

\def\RR{\mathbb{R}}

\newcommand{\eq}[1]{\begin{equation}\label{#1}}
\newcommand{\en}{\end{equation}}
\newcommand{\eqs}{\begin{equation*}}
\newcommand{\ens}{\end{equation*}}

\newenvironment{algor}{\mbox{ } \vspace{10pt} \hrule \mbox{ } \newline 
\noindent \vspace{3pt} \sl }{ 
\mbox{ } \vspace{2pt} \hrule \mbox{ } \vspace{10pt} \mbox{ } }
 
\def\betab{
\begin{tabbing}
xxx\=xxx\=xxxx\=xxxx\=xxxxxxxxx\=xxxxxxxxx\=xxx\=\kill}
\def\entab{\end{tabbing}}

\def\ip{\mbox{ip}} 
                                          
\def\argmax{\mbox{argmax}}

\newtheorem{lemma}{Lemma}
\newtheorem{theorem}{Theorem}
\graphicspath{{FIGS/}}
\newenvironment{proof}{\begin{trivlist}
		       \item[]\hspace{0.3cm}{\bf Proof.}
		       \hspace{0cm} }{\hfill $\Box$
		       \end{trivlist}}

\title{Sampling and multilevel coarsening algorithms for fast matrix 
approximations\thanks{This work was supported by NSF under  grant
       NSF/CCF-1318597.}} 

\author{Shashanka Ubaru\thanks{IBM T. J. Watson Research Center, Yorktown Heights, NY10598.
              This work was undertaken when the author was a student at University of Minnesota. 
              Email: {Shashanka.Ubaru@ibm.com}
}
\and Yousef Saad\thanks{Department of Computer Science and Engineering,
                University of Minnesota at Twin Cities, MN 55455.
                Email: saad@umn.edu.}}

\date{} 

\begin{document} 

\maketitle

 \begin{abstract}

  This paper addresses matrix approximation problems for matrices that
  are large, sparse  and/or that are representations  of large graphs.
  To  tackle these problems,  we consider  algorithms  that are  based
  primarily on  coarsening techniques,  possibly combined  with random
  sampling.   A  multilevel  coarsening  technique  is  proposed which
  utilizes  a  hypergraph  associated with the data  matrix  and  a  graph
  coarsening strategy   based on  column matching.   We  consider  a number  of
  standard applications of  this technique as well as a  few new ones.
  Among  standard applications  we first  consider the  problem of
  computing the \emph{partial SVD} for which a combination of sampling
  and coarsening yields significantly improved SVD results relative to
  sampling alone.  We also consider the \emph{Column Subset Selection}
  problem,  a  popular low  rank  approximation  method used  in  
  data-related applications, and show how multilevel coarsening can
  be adapted for  this problem. Similarly, we consider  the problem of
  \emph{graph sparsification}  and show how coarsening  techniques can
  be  employed  to solve  it. We also establish theoretical
  results  that  characterize  the approximation error 
  obtained and 
  the quality  of  the
  dimension reduction  achieved by  a coarsening  step, when  a proper
  column  matching strategy  is  employed.   
  Numerical  experiments illustrate  the
  performances of the methods in a few applications.
  
 \end{abstract}




\section{Introduction} 
Modern applications involving data often  rely on very large datasets,
but  in  most  situations  the  relevant information  lies  in  a  low
dimensional  subspace.   In  many  of  these  applications,  the  data
matrices are  sparse and/or are  representations of large  graphs.  In
recent  years, there  has been  a surge  of interest  in approximating
large matrices  in a variety  of different ways,  such as by  low rank
approximations~\cite{drineas2006fast,halko2011finding,mahoney2011randomized},
graph  sparsification~\cite{spielman2011graph,kapralov2014single}, and
compression~\cite{kumar2012sampling}.   Methods  to  obtain  low  rank
approximations  include  the   partial  singular  value  decomposition
(SVD)~\cite{halko2011finding}  and   Column  Subset   Selection  (CSS)
\cite{boutsidis2009improved}.  Efficient  methods have  been developed
to  compute   the  partial   SVD~\cite{Saad-book3,golub2012matrix},  a
problem that has been studied for a few decades.  However, traditional
methods for partial SVD computations  cannot cope with very large data
matrices.  Such  datasets prohibit even  the use of  rather ubiquitous
methods    such    as    the    Lanczos    or    subspace    iteration
algorithms~\cite{Saad-book3,saad2016analysis}, since  these algorithms
require  consecutive  accesses to  the  whole  matrix multiple  times.
Computing such matrix  approximations is even harder  in the scenarios
where the matrix under consideration  receives frequent updates in the
form of new columns or rows.

Much recent attention has been devoted to a class of `randomization' 
techniques~\cite{drineas2004clustering,drineas2006fast,halko2011finding}
whereby an approximate partial SVD is obtained from 
a small randomly sampled subset of the matrix, or possibly a few subsets.
Several randomized embedding and sketching methods have also been
proposed~\cite{halko2011finding,woodruff2014sketching}.
These randomization techniques are  well-established (theoretically)
and are proven to give good
results in some situations, see~\cite{mahoney2011randomized}
for a review. 
In this paper we will consider random sampling
methods as a first step  to down sample very large datasets when
necessary. However, randomized methods by themselves 
 can be suboptimal in many situations since they do not exploit 
 available information or the redundancies  in the matrix.
For example, many sampling methods only consider column norms, 
and embedding and sketching methods
 are usually independent of the input matrix.
For this reason, such scheme are  often 
termed  ``data-oblivious''~\cite{ailon2010faster}. 
One of the goals of this work is to  show that  multilevel graph coarsening,
a technique that is often used in the different context of graph 
partitioning~\cite{hendrickson1995multilevel},  can provide superior 
alternatives to randomized sampling, at a moderate cost.

Coarsening  a  graph  (or  a hypergraph)  $G=(V,E)$  means  finding  a
`coarse'   approximation  $\bar{G}=(\bar{V},\bar{E})$   to  $G$   with
$|\bar{V}| < |V|$,  which is a reduced representation  of the original
graph $G$, that retains as much of the structure of the original graph
as  possible.  Multilevel  coarsening   refers  to  the  technique  of
recursively coarsening  the original graph  to obtain a  succession of
smaller  graphs  that approximate  the  original  graph $G$. The
problem of graph and hypergraph coarsening has been extensively studied
  in   the   literature, see, e.g.,~\cite{hendrickson1995multilevel,
karypis1998fast,ca:hyppartition99,kk:kwaypart00}.


{These techniques are more accurate than down-sampling with column
norm   probabilities~\cite{drineas2006fast}, and yield comparable results to the popular leverage  scores based sampling~\cite{drineas2008relative}
which  can be  expensive, but more accurate than  column norm  based sampling. 
Moreover,  coarsening will be inexpensive compared to these sampling methods for large (size $n>10^5$) 
and sparse ($\nz(A)=O(n)$) matrices. For really large ($n>10^6$) matrices, a  typical algorithm  would first  perform uniform down-sampling  of the matrix to  reduce the  size  of  the  problem  and then  utilize  a
multilevel coarsening  technique for computing an  approximate partial
SVD of the reduced matrix.}

The second low rank approximation  problem considered in this paper is
the       \emph{column         subset         selection
  problem}~\cite{boutsidis2009improved,wang2015empirical}   (CSSP)  or
CUR   decomposition~\cite{mahoney2009cur,drineas2008relative}.   Here,
the goal is to select a  subset of columns that best represent the
given matrix spectrally, i.e., with respect to spectral and Frobenius norms.  Popular
methods   for the  CSSP   use the  leverage  score   sampling  measure for
sampling/selecting  the   columns.   Computing  the   leverage  scores
requires  a partial  SVD  of the  matrix and  this  may be  expensive,
particularly for large  matrices and when the (numerical)  rank is not
small.  We will see  how graph coarsening techniques can
be adapted for column subset selection. The coarsening approach
is an  inexpensive alternative for  this problem and it performs  well in
many situations as the experiments will show.

The   third    problem   we   consider   is    that   of   \emph{graph
  sparsification}~\cite{Faloutsos05,spielman2011graph,kapralov2014single}.   Here,
given  a  large (possibly  dense)  graph  $G$,  we  wish to  obtain  a
sparsified graph  $\tilde{G}$ that has significantly  fewer edges than
$G$ but  still maintains important  properties of the  original graph.
Graph sparsification allows one to  operate on large (dense) graphs $G$
with a reduced  space  and  time complexity.   In  particular,  we  are
interested in spectral sparsifiers,  where the Laplacian of $\tilde{G}$
spectrally         approximates        the         Laplacian        of
$G$~\cite{spielman2011spectral,kapralov2014single,woodruff2014sketching}. That
is, the  spectral norm  of the  Laplacian of  the sparsified  graph is
close to the  spectral norm of the Laplacian of  $G$, within a certain
additive or multiplicative  factor.   Such  spectral   sparsifiers  can help
approximately solve  linear systems with   the Laplacian of $G$  and to
approximate effective  resistances, spectral clusterings,  random walk
properties, and  a variety of  other computations.  We will again  show how
  graph  coarsening  can  be  adapted for the task of  graph
sparsification.



The  outline of  this  paper is  as follows.   Section~\ref{sec:basic}  describes   a few   existing  
algorithms to compute (low-rank)  approximations of matrices. 
Graph  coarsening  and  multilevel   algorithms  are
presented  in section~\ref{sec:coarsen}.   In  particular,  we present  a
hypergraph  coarsening technique  based on  column matching and 
discuss  methods  to improve  the  SVD  obtained from  randomization and
coarsening  methods.  In  section~\ref{sec:analysis},  we establish  
theoretical error bounds for the  coarsening method. 
Section~\ref{sec:appl}
discusses  a  few  data  related  applications  of  graph  coarsening.
Numerical experiments  illustrating the performances of  these methods
in a variety of applications are presented in section~\ref{sec:expt}.

\section{Background}\label{sec:basic} 
In this section,  we review three well-known classes of  methods to 
calculate the partial  SVD of matrices. The first class  is based on
randomized  sampling. We  also  consider solving the   column subset  selection
(CSSP)  and   graph  sparsification problems  using  randomized   sampling,  in
particular  based on leverage  score  sampling.   The second class  is  the set of methods
based on subspace  iteration,  and the  third   is  the set  of
SVD-updating    algorithms~\cite{zha1999updating,vecharynski2014fast}.
We consider the latter two classes  of methods as tools to improve the
results obtained  by sampling  and coarsening  methods. Hence,  we are
particularly interested  in the situation  where the matrix  $A$ under
consideration receives  updates in  the form of  new columns.  In fact
when combined with the multilevel algorithms to be  discuss in
section~\ref{sec:coarsen}, these updates are  not small since the number
of columns can double.

\subsection{Random sampling} 
Randomized algorithms have attracted much attention in recent years due to their
broad applications as well as  to the related theoretical results that
have been shown, which are independent of the matrix spectrum. Several
`randomized embedding' and `sketching'  methods have been proposed for
low    rank   approximation    and   for    computing   the    partial
SVD~\cite{frieze2004fast,martinsson2006randomized,liberty2007randomized,halko2011finding,ubaru2017low}.
Drineas et  al.~\cite{drineas2004clustering,drineas2006fast} presented
the randomized subsampling algorithms, where a submatrix consisting of
a few columns  of the original matrix is randomly  selected based on a
certain probability  distribution.  Their  method samples  the columns
based on  column norms. Specifically, column $i$ of 
a given  matrix $A\in\RR^{m\times  n}$ is selected with 
probability $p_i$ given by
\[
 p_i=\frac{\beta\|A^{(i)}\|_2^2}{\|A\|_F^2},
\]
where $\beta<1$ is a positive constant
and $A^{(i)}$ is the $i$-th column of $A$.
Using the above distribution, $c$ columns are selected and the 
subsampled matrix $C$ is formed by scaling the columns by $1/\sqrt{cp_i}$.
The SVD of $C$ is then computed.
The approximations obtained by this randomization method will
yield reasonable results only when there is a  sharp
decay in the singular value spectrum, and computing all the column norms and then sampling can become
 expensive for really large matrices.

\subsection{Column Subset Selection}
Another well-known dimensionality reduction method  considered in 
this paper is the column subset selection 
(CSSP)~\cite{boutsidis2009improved} technique.
 If  a subset of the rows is also selected, then the method leads to the
 CUR decomposition~\cite{mahoney2009cur}. These methods can be viewed as 
extensions of the  randomized sampling based algorithms.
Let  $A\in\RR^{m\times n}$  be a large data matrix whose columns 
we wish to select and  suppose $V_k$ is a matrix whose columns are 
the  top $k$ right singular vectors of $A$. 
Then, the leverage score of the $i$-th column of $A$ is given by
\[
 \ell_i=\frac{1}{k}\|V_k(i,:)\|_2^2,
\]
the  scaled square  norm  of  the $i$-th  row  of  $V_k$.  In
leverage score  sampling, the  columns of $A$  are then sampled  using the
probability  distribution  $p_i=\min\{1,\ell_i\}$.  The  most  popular
methods  for CSSP  involve  the use  of this  leverage  scores as  the
probability            distribution             for            columns
selection~\cite{drineas2008relative,boutsidis2009improved,mahoney2009cur,boutsidis2017optimal}.
Greedy subset  selection algorithms have  also been proposed  based on
the        right         singular        vectors         of        the
matrix~\cite{paschou2007intra,avron2013faster}.    Note that these
methods  may be  expensive when the rank $k$ is not small
 since they require  the top  $k$
singular vectors.   
An alternative solution to the CSSP 
is to use the coarsened graph consisting of 
the columns obtained by graph coarsening.
We will see how this technique compares with standard ones.

 \subsection{Graph Sparsification}
 Given a  large graph $G=(V,E)$ with  $n$ vertices, we wish  to find a
 sparse approximation to this graph that preserves certain information
 of     the     original     graph    such     as     the     spectral
 information~\cite{spielman2011spectral,kapralov2014single},
 structures       like       clusters        within       in       the
 graph~\cite{Faloutsos05,Leskovec06},  etc.  Sparsifying  large graphs
 has  several  computational  advantages  and  has  hence  found  many
 applications~\cite{Faloutsos05,Leskovec06,Satuluri11,spielman2011graph,spielman2011spectral}.

 Let  $B\in\RR^{{n\choose 2}\times  n}$ be  the vertex  edge incidence
 matrix  of the  graph $G$,  where  the $e$-th row  $b_e$ of  $B$ for  edge
 $e=(u,v)$ of  the graph has a  value $\sqrt{w_e}$ in columns  $u$ and
 $v$, and zero  elsewhere, in which $w_e$ is the weight  of the edge. The
 corresponding Laplacian of the graph  is then given by $K=B^TB$.  The
 spectral  sparsification   problem  involves  computing   a  weighted
 subgraph $\tilde G$  of $G$ such that if $\tilde  K$ is the Laplacian
 of  $\tilde G$,  then $x^T\tilde{K}x$  is  close to  $x^TKx$ for  any
 $x\in\RR^n$.  A number of methods  have been proposed for sparsifying
 graphs,                                                           see
 e.g.,~\cite{spielman2011graph,spielman2011spectral,kapralov2014single,woodruff2014sketching}.
 One such approach performs row sampling of the matrix $B$ using
 the  leverage score  sampling criterion~\cite{kapralov2014single}.  
 Considering
 the SVD  of $B=U\Sigma V^T$, the  leverage scores $\ell_i$ for  a row
 $b_i$ of $B$  can be computed as  $\ell_i=\|u_i\|_2^2\leq1$ using the
 rows  of  $U$.  This  leverage  score  is  related to  the  effective
 resistance  of edge  $i$~\cite{spielman2011graph}.   By sampling  the
 rows of  $B$ according  to their  leverage scores  it is  possible to
 obtain      a       matrix      $\tilde      B$,       such      that
 $\tilde{K}=\tilde{B}^T\tilde{B}$  and  $x^T\tilde{K}x$  is  close  to
 $x^TKx$ for  any $x\in\RR^n$.  In section~\ref{sec:coarsen},  we show
 how the rows of $B$ can we selected via coarsening.

\subsection{Subspace iteration} 

\begin{algorithm}[t!]
\caption{Incremental Subspace Iteration\label{alg:subsit}  }
\begin{algorithmic}
  \State {\bf Input:}  $U_s$, from  $A_s = U_s \Sigma_s V_s^T$,
   coarsened matrix of $A$, $A_t$ (larger coarsened matrix), 
  and no. of iterations 
  $iter$.  
  \State {\bf Output:} Updated SVD $[U,\Sigma,V]$ of $A$.
\State Start: $U = U_s$
\For{$i=1: iter$} 
\State $ V = A_t^T U $
\State $ U = A_t V $
\State $U := qr(U,0); \quad V := qr(V,0)$;
\State $S = U^T A_t V$ 
 
\State $[R_U, \Sigma, R_V]                                          
= \text{svd}( S )$; $U := UR_U$; $V := VR_V$ 
\State\Comment{Above two steps used only if we 
require current estimate of singular vectors and values}
\EndFor 
\end{algorithmic}
\end{algorithm}

Subspace iteration is a well-established  method used for solving eigenvalue  and
singular value problems~\cite{golub2012matrix,Saad-book3}. We review this algorithm as it will
be exploited later as  a tool 
to improve SVD results obtained by sampling and coarsening methods.
A known advantage  of the subspace iteration 
 algorithm  is that  it is very
robust and that it tolerates  changes in the matrix~\cite{saad2016analysis}.
This is important in our context.  Let us consider 
a general matrix $A\in\RR^{m\times n}$, not necessarily associated with a  graph. 
The subspace iteration algorithm  can easily be
adapted to the situation where a previous SVD is available for
a smaller version of $A$ with fewer rows or columns, obtained by subsampling or coarsening for 
example. 
Indeed, let $A_s$ be a column-sampled
version  of $A$. In matlab notation we represent this as 
$A_s = A(:,J_s)$ where $J_s$ is a subset of the column index $[1 : n]$.
Let $A_t$ be another subsample of $A$, where we assume that $J_s \subset J_t$.
Then if $A_s = U_s \Sigma_s V_s^T$, we can perform a few steps of subspace 
iteration updates as shown in Algorithm~\ref{alg:subsit}.
Note that the last two steps in the algorithm which diagonalize $S$ are needed
only if we require the current 
estimate of singular vectors and values.



\subsection{SVD updates from subspaces}\label{sec:SVDupdate} 
A well known algorithm for updating the SVD is the 
`updating algorithm' of Zha and Simon~\cite{zha1999updating}. 
Given a matrix $A\in\RR^{m\times n}$ and its partial SVD $[U_k,\Sigma_k,V_k]$,
the matrix $A$ is updated by adding
columns $D$ to it, resulting in  a new matrix $A_D = [A, D]$,
where $D\in\RR^{m\times p}$ represents the added columns.
The algorithm then first computes 
\begin{equation}\label{eqn:qr_doc} 
(I - U_k U_k^T) D = \hat U_p R,  
\end{equation} 
 the truncated (thin) QR decomposition of $(I - U_k U_k^T) D$, where 
$\hat U_p \in \mathbb{R}^{m \times p}$ has orthonormal columns and $R \in \mathbb{R}^{p \times p}$ is  
upper triangular.  
Given~(\ref{eqn:qr_doc}), one can observe that   
\begin{equation}\label{eqn:relation_docs} 
A_D = [U_k, \ \hat U_p] 
H_D 
\left[ 
\begin{array}{cc} 
 V_k & 0 \\  
 0   & I_p  
\end{array} 
\right]^T \;, \ 
H_D = \left[ 
\begin{array}{cc} 
 \Sigma_k & U_k^T D \\  
 0        & R  
\end{array} 
\right] \;, 
\end{equation} 
where $I_p$ denotes the $p$-by-$p$ identity matrix.  
Thus, if $\Theta_k$, $F_k$, and $G_k$ are the matrices corresponding to the $k$ dominant singular values of  
$H_D \in \mathbb{R}^{(k+p) \times (k+p)}$ 
and their left and right singular vectors, respectively,  
then the desired updates $\tilde \Sigma_k$, $\tilde U_k$, and $\tilde V_k$ are given by 
\begin{equation}\label{eqn:upd_docs} 
\tilde \Sigma_k = \Theta_k, \ \tilde U_k = [U_k, \ \hat U_p] F_k, \ \mbox{and} \ \tilde V_k =  
\left[ 
\begin{array}{cc} 
 V_k & 0 \\  
 0   & I_p  
\end{array} 
\right] G_k \;.  
\end{equation}    

The QR decomposition  in the first step  eq.~\eqref{eqn:qr_doc} can be
expensive when  the updates are large  so an improved version  of this
algorithm  was   proposed  in~\cite{vecharynski2014fast}   where  this
factorization  is replaced  by a  low rank  approximation of  the same
matrix.  That is, for a rank $l$, we compute a rank-$l$ approximation,
$ (I -  U_k U_k^T) D =  X_lS_lY_l^T.  $ Then, the matrix  $H_D$ is the
update equation~\eqref{eqn:upd_docs} will be
\[
 H_D=\begin{bmatrix}
     \Sigma_k & U_k^T D \\  
 0        & S_lY_l^T
    \end{bmatrix}
\]
with $\bar{U}=[U_k,X_l]$. The idea is  that the update $D$ will likely
be low rank outside the previous  top $k$ singular vector space. Hence
a  low rank  approximation of  $  (I -  U_k U_k^T)  D$ suffices,  thus
reducing the cost.

In  applications  involving  low rank approximations, the rank $k$ 
will be typically much smaller than $n$, and it can be
 inexpensively estimated using  recently proposed
methods~\cite{ubaru2016fast,ubaru2017fast}.

\section{Coarsening}\label{sec:coarsen}

\begin{figure}[tb]
\centerline{\includegraphics[width=0.55\textwidth]{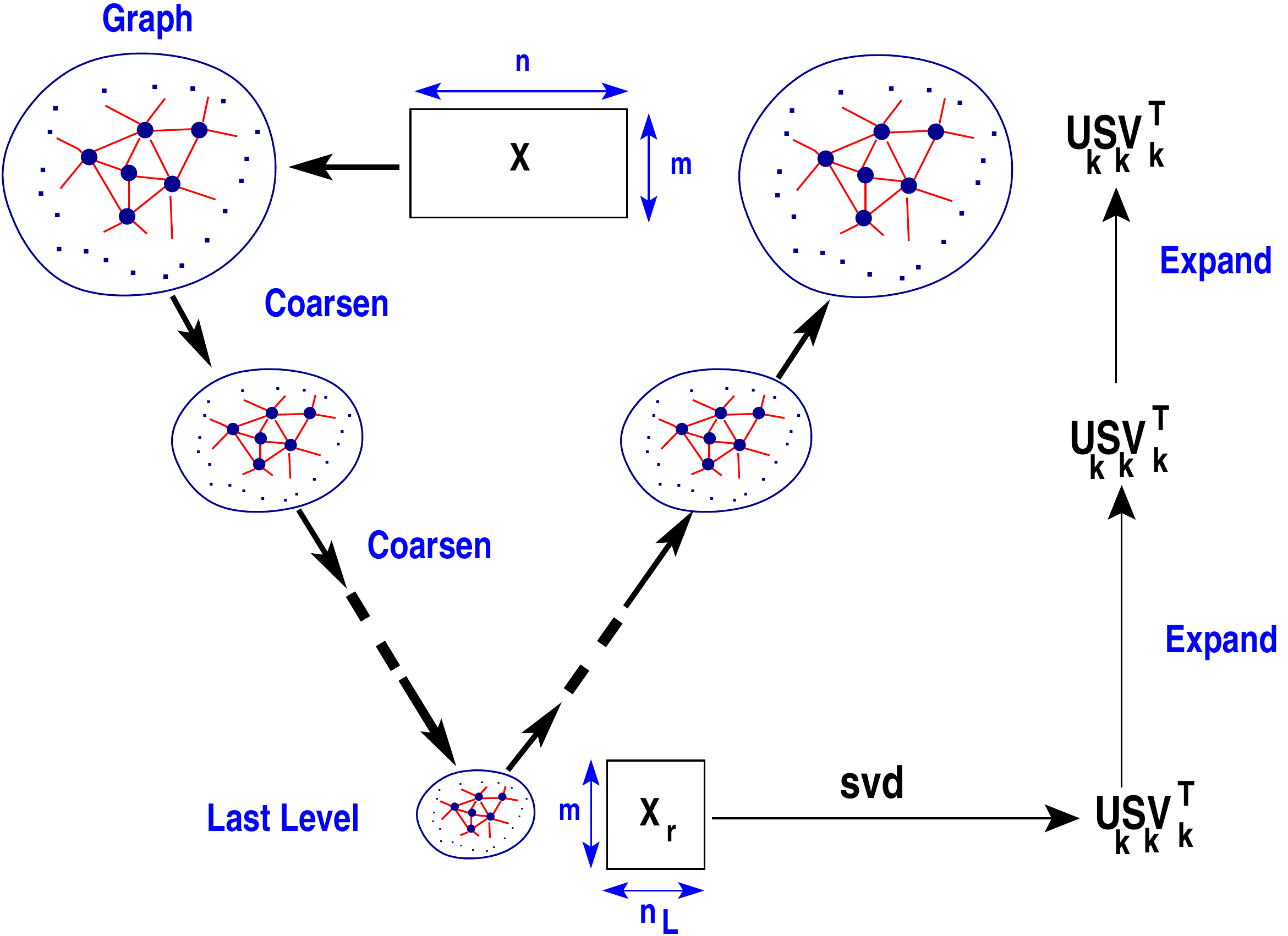} 
\includegraphics[width=0.4\textwidth]{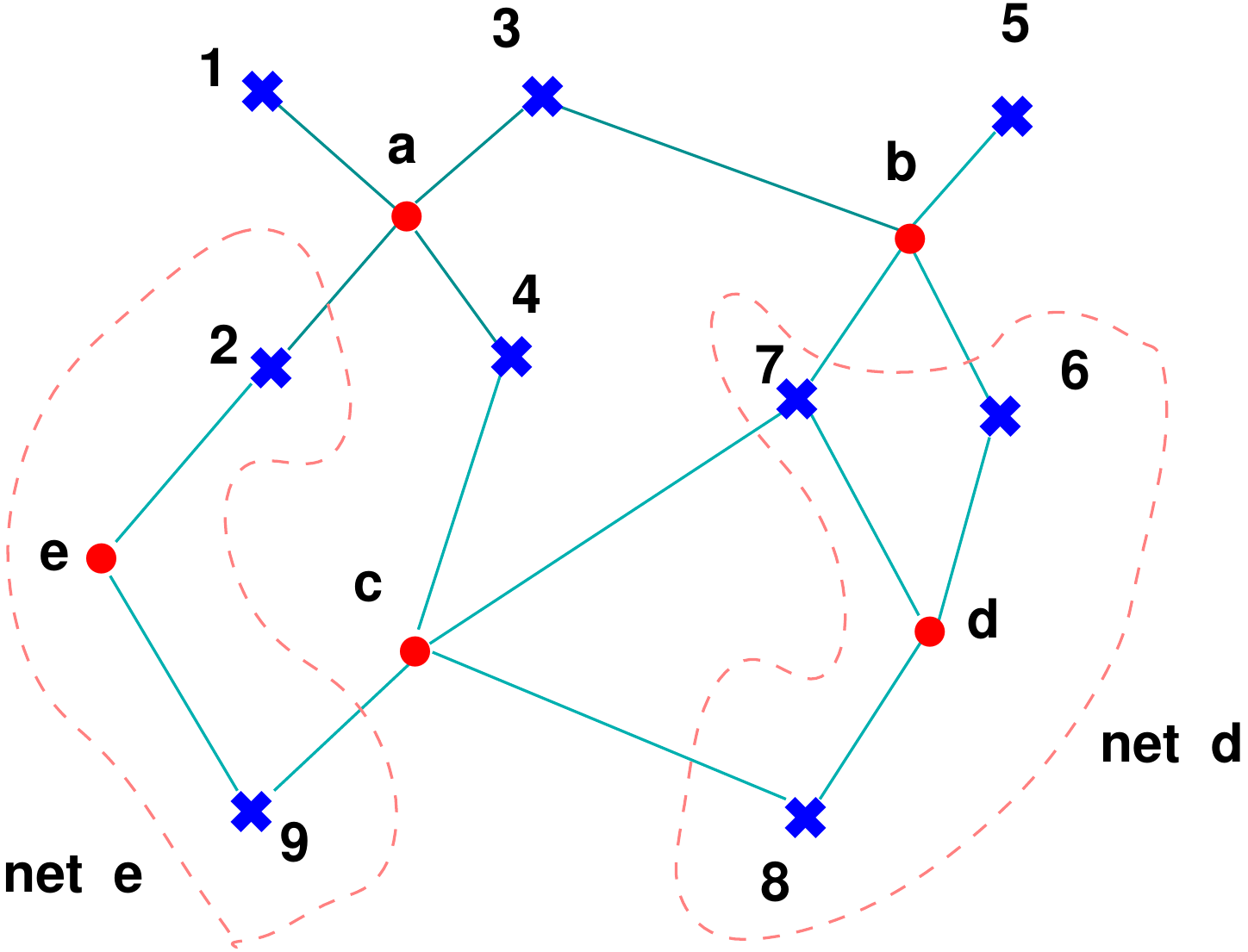}}
\caption{Left: Coarsening / uncoarsening procedure; Right : A sample hypergraph.}\label{fig:coarsen} 
\end{figure} 

As discussed in the previous section,
randomized sampling methods can be an effective approach
for  computing the partial SVD 
in certain  situations, for example, when  there is a good  gap in the
spectrum or if  there is a sharp  spectral decay. An alternative  approach to
reduce  the  matrix  dimension,  particularly when  the  matrices  are
associated with graphs,  is to coarsen the data with the help of
 graph coarsening,
perform all computations  on the resulting 
reduced size matrix, and  then project back
to the  original space. Similar  to the  idea of sampling  columns and
computing the  SVD of  the smaller sampled  matrix, in  the coarsening
methods,  we compute  the SVD  from  the matrix  corresponding to  the
coarser data. It is also possible to then wind back up and correct the
SVD   gradually,  in   a  way   similar  to   V-cycle  techniques   in
multigrid~\cite{hss:mlevel-08},     this     is     illustrated     in
Figure~\ref{fig:coarsen}(left).                 See,               for
example~\cite{zhou2006learning,hss:mlevel-08,fang2010multilevel,HRFangEtAlcikm2}
for  a few  illustrations  where coarsening  is  used in  data-related
applications.

Before coarsening,  we first  need to build  a graph  representing the
data. This  first step  may be  expensive in some  cases but  for data
represented by sparse  matrices, the graph is available  from the data
itself in  the form  of a  standard graph or  a hypergraph.  For dense
data, we need to set-up  a similarity graph, see~\cite{cfsKnn08} for a
fast algorithm  to achieve  this. 
This paper will  focus on sparse data such as  the data sets available
in  text mining,  gene expressions  and multilabel  classification, to
mention a few  examples, and advocate the use of coarsening technique for dimension
reduction.  In such cases, the data  is represented by a
(rectangular)  sparse   matrix  and  it is most convenient  to  use
hypergraph models~\cite{zhou2006learning} for coarsening.

\subsection{Hypergraph Coarsening}
\label{sec:model}

\begin{algorithm}[tb!]
\caption{Hypergraph coarsening by column matching.\label{alg:coarsening}}
\begin{algorithmic}
  \State {\bf Input:}    $A\in\RR^{m\times n}$, $\eps\in(0,1)$. 
        
  \State {\bf Output:} Coarse matrix $C\in\RR^{m\times c}$.
  \State $Idx:=\{1,\dots,n\}$ \Comment{(Set of unmatched vertices)}
    \State Set $\ip[k]:= 0$ for $k=1,\dots,n$, and $\ell=1$. 
 \Comment{(Initialization)} 
 
  \Repeat
    \State Randomly pick $i\in Idx$; $Idx:= Idx-\{i\}$.
      \ForAll {$j$ with $a_{ij}\neq 0$} \Comment{(*)}
        \ForAll {$k$ with $a_{jk}\neq 0$}
        \State {$\ip[k]:= \ip[k]+a_{ij}a_{jk}$. } 
      \EndFor
    \EndFor
    \State {$j := \argmax\{\ip[k]:k\in Idx\}$}
    \State {$csq\theta = \frac{\ip[j]^2}{\|a^{(i)}\|^2\|a^{(j)}\|^2}$.}
   \If {[ $(csq\theta\geq\frac{1}{1+\eps^2})$]} \Comment{Match only if the angle satisfies the condition}
       \State{$c^{(\ell)}:= \sqrt{1+csq\theta}a^{(i)}$}. \Comment{The denser of columns $a^{(i)}$ and $a^{(j)}$}
      \State{$Idx:= Idx-\{j\}; \ell=\ell+1$.}
      \Else
       \State{$c^{(\ell)}:= a^{(i)}$}.
        \State{$ \ell=\ell+1$.}
   \EndIf
   \State{Reset nonzero values of ip to zero}\Comment{(A sparse operation)} 
    \Until {$Idx=\emptyset$}
\end{algorithmic}
\end{algorithm}
Hypergraphs extend the classical notion of graphs.
A hypergraph $H=(V,E)$
consists of a set of vertices $V$ and a set of
hyperedges $E$~\cite{ca:hyppartition99,zhou2006learning}.
In a standard graph an edge connects two
vertices, 
whereas a hyperedge may connect an arbitrary subset of vertices.
A hypergraph $H=(V,E)$ can be canonically represented by
a sparse matrix $A$,
where the vertices in $V$ and hyperedges (nets) in $E$ are represented
by the columns and rows of $A$, respectively.
This is called the {\em row-net model}.
Each hyperedge, a row of $A$, connects the vertices, i.e., the columns,
whose corresponding entries in that row are non-zero.
An  illustration is provided in Figure~\ref{fig:coarsen}(Right),
where $V=\{1,\dots,9\}$ and $E=\{a,\dots,e\}$ with
$a=\{1,2,3,4\}$, $b=\{3, 5, 6, 7\}$, $c=\{4, 7, 8, 9\}$, $d=\{6,7,8\}$,
and $e=\{2,9\}$. 

%

Given a  (sparse) data set of $n$ entries  in $\RR^m$ represented
by a matrix  $A \in \RR^{m\times n}$, we can  consider a corresponding
hypergraph $H=(V,E)$ with vertex set  $V$ representing the columns
of  $A$. 
Several  methods   exist  for   coarsening  hypergraphs,   see,  e.g.,
\cite{ca:hyppartition99,kk:kwaypart00}.     Here,   we    consider   a
hypergraph coarsening  based on column  matching, which is  a modified
version   of  the   {\em   maximum-weight   matching}  method,   e.g.,
\cite{ca:hyppartition99,dbhbc:hypergraph06}.   The  modified  approach
follows the  maximum-weight matching method and  computes the non-zero
inner  product $\langle  a^{(i)},a^{(j)}\rangle$ between  two vertices
$i$ and $j$, i.e.,  the $i$-th and $j$-th columns of  $A$. Note that the
inner  product between  vectors is  related to  the angle  between the
vectors,                         i.e.,                        $\langle
a^{(i)},a^{(j)}\rangle=\|a^{(i)}\|\|a^{(j)}\|\cos\theta_{ij}$.     The
proposed coarsening strategy is to  match two vertices, i.e., columns,
only   if  the   angle  between   the   vertices  is   such  that,   $
\tan\theta_{ij}\leq \eps$, for a constant $0<\eps<1$.  Another feature
of  the  proposed algorithm  is  that  it  applies  a scaling  to  the
coarsened  columns in  order  to  reduce the  error.   In summary,  we
combine two columns $a^{(i)}$ and  $a^{(j)}$ if the angle between them
is such that, $\tan\theta_{ij}\leq \eps$.   We replace the two columns
$a^{(i)}$ and $a^{(j)}$ by
\[c^{(\ell)}=\left(\sqrt{1+\cos^2\theta_{ij}} \right)a^{(i)}\]
where   $a^{(i)}$ is replaced by 
$a^{(j)}$ if the latter has   more  nonzeros.
This minor modification provides  some control over the coarsening procedure 
using the parameter $\eps$
and, more importantly, it helps establish  theoretical results for 
the method, see section~\ref{sec:analysis}.

The vertices can be visited in a random order,
or in the `natural' order in which they are  listed.
For each unmatched vertex $i$,
all the unmatched neighbor vertices $j$ are explored 
and the inner product between $i$ and each $j$ is computed.
This typically requires the data structures of $A$ and its transpose,
in that a fast access to rows and columns is required.
The vertex $j$ with the highest non-zero inner product
$\langle a^{(i)},a^{(j)}\rangle$ is considered and if the angle between them 
is such that  $\tan\theta_{ij}\leq \eps$ (or $\cos^2\theta_{ij}\geq \frac{1}{1+\eps^2})$), 
then $i$ is matched with $j$ and
the procedure is repeated until all vertices have been matched.
Algorithm~\ref{alg:coarsening} provides details on the procedure.

 Computing  the cosine of the angle between  column $i$ and
all other columns is equivalent to computing the $i$-th row of $A^TA$.
{In fact, we only need to compute the upper triangular  part of $A^TA$ since it is symmetric.}  
For  sparse matrices,  the inexpensive computation  of the  inner product  
between the columns used in the algorithm is achieved  by modifying
the  cosine algorithm  in~\cite{saad2003finding} developed  for matrix
blocks detection.
Thus, loop (*) computes all inner products of column $i$ with
all other columns, and accumulates these in the sparse `row' 
$\ip[.]$. This amounts in essence
to computing the $i$-th row of $A^T A$ as the combination of rows:
$\sum_{a_{ij}\ne 0} a_{ij} a_{j,:}$.
As indicated in the line just before the end, resetting $\ip[ . ]$ to zero
is a sparse operation that does not require zeroing out the whole vector 
but only those entries that are nonzero.

The pairing used by the algorithm relies only on the sparsity pattern.
It is clear  that these entries can  also be used to  obtain a pairing
based on  the cosine of the  angles between columns $i$  and $k$.  The
coarse column $c^{(p)}$ is defined as the `denser of columns $a^{(i)}$
and $a^{(j)}$'. In other models the  sum is sometimes used. Note that
the number of  vertices (columns) in $C$ will depend  on 
the redundancy among 
 the data and the $\epsilon$ value chosen, see further discussion in
section~\ref{sec:analysis}.

\paragraph{Computational  Cost}  
We saw earlier that the inner products of a given column $i$ with all 
other columns amounts to computing the nonzero values
of the $i$-th row of the upper triangular part of $A^T A$.  
If we call $Adj_A(i) $ the set of nonzero indices of the $i$-th column of $A$ 
then according to what was said above,
the cost of computing $\ip[.]$ by the algorithm is 
\[
 \sum_{j \ \in \ Adj_A(i), \ j>i } \ | Adj_{A^T} (j) |,
\]
where $|\cdot |$ is the cardinality of the set.
If $\nu_c$ (resp. $\nu_r$) 
is the maximum number of nonzeros in each column (resp. row),
then an upper bound for the above cost is $n \nu_r \nu_c$, which
is the same upper bound as that of  computing  the upper
triangular part  of $A^TA$.  Several simplifications  and improvements
can be added to  reduce the cost. First, we can  skip the columns that
are  already matched.  In this way,  fewer  inner
products are computed  as the  algorithm progresses.
In addition,
since  we only  need the  angle to  be such  that $\tan\theta_{ij}\leq
\eps$, we can reduce the computation cost significantly by stopping
as soon as we encounter a 
column with which the angle is smaller than the threshold. {Hence, this coarsening approach can be
 quite inexpensive for very large sparse matrices, see  section~\ref{sec:expt} for  numerical results.}
Article~\cite{chen2012dense} uses the angle based column matching idea for 
dense subgraph detection in graphs, and describes efficient methods to compute the inner products. 

\subsection{Multilevel SVD computations}


Given a sparse  matrix $A$,
we can use Algorithm~\ref{alg:coarsening} repeatedly with different (increasing) 
$\eps$ values,
to recursively coarsen
the corresponding hypergraph,
and obtain a sequence of sparse matrices $A_1,A_2,\dots,A_s$ with $A_0=A$,
where $A_i$ corresponds to the coarse graph $H_i$ of level $i$
for $i=1,\dots,s$, and $A_s$ represents the lowest level graph $H_s$.
This provides a reduced size matrix which will likely
be a good representation of the original data. Note that, recursive coarsening
will be inexpensive since the inner products required in the further levels
are already computed in the first level of coarsening.

 In the multilevel framework of hypergraph coarsening
 we apply the matrix approximation method, say {low rank approximation using the SVD}, to
 the coarsened data matrix $A_s\in \RR^{m\times n_s}$
 at the lowest level,
 where $n_s$ is the number of columns at the coarse level $s$ ($n_s<n$).
 {A low-rank matrix approximation can be viewed as
 a linear projection of the rows into a lower dimensional space.
 In other words, we have a projected matrix $\hat{A}_s\in \RR^{d\times n_s}$ ($d<m$) from the coarse matrix $A_s$
 which preserves certain features of $A_s$ (say spans the top $k$ singular vectors of $A_s$).
Suppose we apply the same linear projection to $A\in\RR^{m\times n}$ (the original matrix),
resulting in $\hat{A} \in \RR^{d\times n}$ ($d<m$), we can expect that
 $\hat{A}$ preserves certain features of $A$ (approximately spans the top $k$ singular vectors of $A$).
 Hence, we can obtain a reduced representation $\hat{A}\in \RR^{d\times n}$ ($d<m$) of
 the original data $A$ using this linear projection.}
   The procedure is illustrated in Figure~\ref{fig:coarsen} (left).
 A multilevel randomized sampling technique is discussed 
 in~\cite{drineas2006fast}. 
 Another strategy for reducing the matrix dimension is to mix the two 
techniques: 
 Coarsening may still be exceedingly 
expensive for some types of large data where there is no immediate graph available
to exploit for coarsening. In this case, a good strategy would be
to downsample first using the randomized methods (say uniform sampling), 
then construct a graph and coarsen it. In section~\ref{sec:expt}, we 
 compare the SVDs obtained  from pure randomization method (column norm sampling)
against those obtained from coarsening and also  a combination of
randomization (uniform sampling) and coarsening.

\subsection{CSSP and graph sparsification}
The multilevel coarsening technique just  presented can be applied for
the column  subset selection problem (CSSP)  as well as for  the graph
sparsification problem.  We can use  Algorithm~\ref{alg:coarsening} to
coarsen the  matrix, and this is equivalent to selecting columns  of the
matrix. The  only required modification in  the algorithm is  that the
columns selected are  no longer scaled.  
The coarse matrix  $C$ contains a subset  of the 
columns of the original matrix $A$ and the analysis will show that
 it is a faithful representation of the original $A$.
 
 For graph sparsification, we can apply the coarsening procedure on {the vertex edge incidence matrix $B$
 corresponding to the graph $G$. That is, we coarsen the rows of $B$ (by applying Algorithm~\ref{alg:coarsening} to $B^T$), and obtain a matrix $\tilde{B}$ with fewer rows,
 yielding us a graph $\tilde{G}$ with fewer edges.}
 {From the analysis in section~\ref{sec:analysis}, we can show that this coarsening strategy will indeed 
 result in a spectral sparsifier, i.e., we can show that  $x^T\tilde{B}^T\tilde{B}x$ will be close to $x^TB^TBx$ for any vector $x :\|x\|=1$.}
 Since we achieve sparsification via matching, the structures such as clusters
 within the original graph are also  preserved.

\subsection{Incremental SVD} 

Next, we explore combined algorithms with the goal of  improving the randomized
sampling and coarsening SVD results significantly. 
The typical overall algorithm which we call the Incremental SVD  
algorithm, will start with a sampled/coarsened matrix $A_s$ using 
one of the randomized methods or multilevel coarsening discussed above,
and compute its partial SVD.
We then add columns of $A$ to $A_s$, and use Algorithm~\ref{alg:subsit}
or the SVD-update algorithm to update the SVD results.
A version  of this  Incremental algorithm  has been  briefly discussed
in~\cite{gu2015subspace},  where  the  basic randomized  algorithm  is
combined with subspace iteration.

Roughly  speaking,  we  know  that  each  iteration  of  the  subspace
iteration  algorithm will  make the  computed subspace  closer to  the
subspace  spanned by  the  target singular  vectors.   If the  initial
subspace is close to the span  of the actual top $k$ singular vectors,
fewer  iterations  will  be  needed  to  get  accurate  results.   The
theoretical results established in the following section, give an idea
on how the top $k$ singular  vectors for $k<c$ of the coarsened matrix
relate to the  span of the top  $k$ singular vectors of  $A$.  In such
cases, a  few steps  of the  subspace iteration  will then  yield very
accurate results.

For the {SVD updating method discussed in section~\ref{sec:SVDupdate}},  it is  known that the  method performs
well  when  the  updates  are  of  low  rank  and  when  they  do  not
significantly affect  the dominant  subspace, the subspace  spanned by
the top $k$ singular vectors if interest~\cite{vecharynski2014fast}.
Since the random  sampling and the coarsening methods return
a  good approximation  to  the
dominant subspace,  we  can assume that the 
updates in  the incremental SVD are  of low rank, 
and  these updates likely effect the
dominant subspace  only slightly. This is the reason we expect 
that the {SVD updating method} will yield improved  results.

\section{Analysis}\label{sec:analysis}
In this section,  we establish initial theoretical results  for the coarsening
technique  based  on column  matching.   
In  the coarsening  strategy of Algorithm~\ref{alg:coarsening},
 we  combine two  columns $a^{(i)}$ and  $a^{(j)}$  if the
angle  between  them  is  such that  $\tan\theta_{i}\leq  \eps$.   That is, 
 we set 
$c^{(\ell)}=(\sqrt{1+\cos^2\theta_{i}})a^{(i)}$  (or $a^{(j)}$, if it has
more nonzeros) in place of the two columns in the coarsened matrix $C$.
We then have the following result which relates the Rayleigh quotients of 
the coarsened matrix $C$ to that of $A$ (indicates $AA^T\approx CC^T$, spectrally).

\begin{lemma}\label{lemm:RQ}
 Given $A\in\RR^{m\times n}$, let $C\in\RR^{m\times c}$ be the coarsened matrix of $A$ 
obtained by one level of coarsening of $A$ using Algorithm~\ref{alg:coarsening}
with columns $a^{(i)}$ and $a^{(j)}$  matched if $\tan\theta_{i}\leq \eps$,
for $0<\eps<1$. Then,
\begin{equation}
 |x^TAA^Tx-x^TCC^Tx|\leq  3\eps \|A\|_F^2,
\end{equation}
for any $x\in\RR^m$ with $\|x\|_2=1$.
\end{lemma}

\begin{proof} \ 
 Let $(i,j)$ be a pair of matched column indices with $i$ being
  the index of the column that is retained after scaling. We denote by
  $I$ the set of all indices of  the retained columns and $J$ the
  set of the remaining columns.

 We know that
 $x^TAA^Tx=\|A^Tx\|_2^2=\sum_{i=1}^n\langle a^{(i)},x\rangle^2$. 
Similarly, consider
 $x^TCC^Tx=\|C^Tx\|_2^2=\sum_{i\in I}\langle c_i,x\rangle^2
 =\sum_{i\in I}(1+c_i^2)\langle a^{(i)},x\rangle^2$,
where indices $c_i = \cos\theta_i$. We have,
\begin{eqnarray*}
 |x^TAA^Tx-x^TCC^Tx|&=&\left|\sum_{i\in I\cup J}\langle a^{(i)},x\rangle^2-\sum_{i\in I}(1+c_i^2)\langle a^{(i)},x\rangle^2\right|\\
 &{=}& {\left|\sum_{j\in J}\langle a^{(j)},x\rangle^2-\sum_{i\in I}c_i^2\langle a^{(i)},x\rangle^2\right|}\\
 &{\leq}& {\sum_{(i,j)\in I\times J}\left|\langle a^{(j)},x\rangle^2-c_i^2\langle a^{(i)},x\rangle^2\right|}\\
\end{eqnarray*}
where the set $I\times J$ consists of pairs of
 indices  $(i,j)$ that are matched.
Next, we consider an  individual term in the sum.
Let column $a^{(j)}$ be decomposed as follows:
\[
 a^{(j)}=c_ia^{(i)}+s_iw,
\]
where $s_i=\sin\theta_i$ and $w=\|a^{(i)}\|\bar w$ with $\bar w$ a unit vector that is orthogonal 
to $a^{(i)}$. Then,
\begin{eqnarray*}
 |\langle a^{(j)},x\rangle^2-c_i^2\langle a^{(i)},x\rangle^2|&=&
 \left|\langle c_ia^{(i)}+s_iw,x\rangle^2-c_i^2\langle a^{(i)},x\rangle^2\right|\\
 &=&\left|c_i^2\langle a^{(i)},x\rangle^2+2c_is_i\langle a^{(i)},x\rangle\langle w,x\rangle+s_i^2\langle w,x\rangle^2-c_i^2\langle a^{(i)},x\rangle^2\right|\\
 &=&|\sin2\theta_i\langle a^{(i)},x\rangle\langle w,x\rangle+\sin\theta_i^2\langle w,x\rangle^2|\\\
\end{eqnarray*}

If $t_i=\tan\theta_i$, then 
 $\sin2\theta_i=\frac{2t_i}{1+t_i^2}$.  Using the fact that
$|\langle w,x\rangle|\leq\|a^{(i)}\|\equiv\eta_i$ and $\langle a^{(i)},x\rangle\leq \eta_i$, we get
\begin{eqnarray*}
 |\sin2\theta_i\langle a^{(i)},x\rangle\langle w,x\rangle+\sin\theta_i^2\langle w,x\rangle^2|&\leq&
 \eta_i^2\sin2\theta_i\left[1+\frac{\sin^2\theta_i}{2\sin\theta_i\cos\theta_i}\right]\\
 &=&\eta_i^2\sin2\theta_i\left[1+\frac{\tan\theta_i}{2}\right]\\
 &\leq&\frac{2\eta_i^2 t_i+(\eta_i t_i)^2}{1+t_i^2}\\
 &\leq&2\eta_i^2 t_i+(\eta_i t_i)^2.
\end{eqnarray*}

Now, since  our algorithm combines two columns only if 
$\tan(\theta_i)\leq \eps$ (or $\cos^2\theta \geq 1/(1+\eps^2)$),
we have
\[
 |\langle a^{(j)},x\rangle^2-c_i^2\langle a^{(i)},x\rangle^2|
 \leq 2\eta_i^2\eps+ \eta_i^2\eps^2 \leq 3\eps\eta_i^2
\]
as $\eps<1$. 
Thus, we have
\[
  |x^TAA^Tx-x^TCC^Tx|\leq 3\eps \sum_{i\in I}\|a^{(i)}\|^2\leq3\eps \|A\|_F^2.
\]
 \end{proof}
Note that the above bound will be better in practice because the
last inequality may not be tight. In fact
 improving this result could be an  interesting question to investigate.
We also observe that the above result will hold even if
we consider combining multiple columns which are 
within the angle $\theta=\tan^{-1}(\eps)$
from each other into one in Algorithm~\ref{alg:coarsening}. 
The number of columns $c$ in the coarsened matrix $C$ will depend on the given data 
(the number of pairs of columns that are within the desired angle).
 The above lemma is similar to the results established in~\cite{liberty2013simple} 
 for deterministic matrix sketching. 
It can  be used  to develop the following error bounds.

\begin{theorem}\label{theo:1}
  Given  $A\in\RR^{m\times  n}$,  let  $C\in\RR^{m\times  c}$  be  the
  coarsened matrix of $A$ obtained by one level coarsening of $A$ with
  columns $a^{(i)}$ and $a^{(j)}$ combined if $\tan\theta_{i}\leq
  \eps$,  for $0<\eps<1$.   
Let $H_k$  be the matrix consisting of the top  $k$ left  singular
  vectors of $C$ as columns, for $k\leq c$. Then, we have
\begin{eqnarray}
\|A-H_kH_k^TA\|_F^2&\leq&\|A-A_k\|_F^2+6k\eps\|A\|_F^2\\
\|A-H_kH_k^TA\|_2^2&\leq&\|A-A_k\|_2^2+6\eps\|A\|_F^2,
\end{eqnarray} 
where $A_k$ is the best rank $k$ approximation of $A$.
 \end{theorem}
\begin{proof}

\emph{Frobenius norm error:} First, we prove the Frobenius norm error bound. We can express 
$\|A-H_kH_k^TA\|_F^2$:
\begin{eqnarray}\label{eq:frob1}
\|A-H_kH_k^TA\|_F^2&=&Tr((A-H_kH_k^TA)^T(A-H_kH_k^TA))\\\nonumber
&=&Tr(A^TA-2A^TH_kH_k^TA+A^TH_kH_k^TH_kH_k^TA)\\\nonumber
&=&Tr(A^TA)-Tr(A^TH_kH_k^TA)\\\nonumber
&=&\|A\|_F^2-\|A^TH_k\|_F^2.
\end{eqnarray}
We get the above simplifications using the equalities: $\|X\|_F^2=Tr(X^TX)$ and 
$H_k^TH_k=I$. Let $h^{(i)}$ for $i=1,\ldots,k$ be the columns of $H_k$. Then, the second term in the above equation
is
$\|A^TH_k\|_F^2=\sum_{i=1}^k\|A^Th^{(i)}\|^2$.

From Lemma~\ref{lemm:RQ}, we have for each $i$,
\[
 |\|A^Th^{(i)}\|^2-\|C^Th^{(i)}\|^2|=|\|A^Th^{(i)}\|^2-\sigma_i^2(C)|\leq3\eps \|A\|_F^2,
\]
since $h^{(i)}$'s are the singular vectors of $C$. Summing over $k$ singular vectors, we get
\begin{equation}\label{eq:part1}
 |\|A^TH_k\|_F^2-\sum_{i=1}^k\sigma_i^2(C)|\leq3\eps k \|A\|_F^2.
\end{equation}
From  perturbation theory~\cite[Thm. 8.1.4]{golub2012matrix}, we have
for $i=1,\ldots,c$:
\[
 |\sigma_i^2(C)-\sigma_i^2(A)|\leq \|AA^T-CC^T\|_2 . 
\]
Next, Lemma~\ref{lemm:RQ} implies:
\[
\|AA^T-CC^T\|_2=\max_{x\in\RR^n:\|x\|=1}|x^T(AA^T-CC^T)x|\leq3\eps \|A\|_F^2 . 
\]
Hence, summing over $k$ singular values,
\begin{equation}\label{eq:part2}
 \left|\sum_{i=1}^k\sigma_i^2(C)-\sum_{i=1}^k\sigma_i^2(A)\right|\leq3\eps k \|A\|_F^2.
\end{equation}
Combining \eqref{eq:part1} and \eqref{eq:part2}, we get
\[
 \left|\|A^TH_k\|_F^2-\sum_{i=1}^k\sigma_i^2(A)\right|\leq6\eps k \|A\|_F^2.
\]
Along  with~\eqref{eq:frob1} this relation
 gives us  the Frobenius norm error bound,
since $\|A\|_F^2-\sum_{i=1}^k\sigma_i^2(A)=\|A-A_k\|_F^2$.

\emph{Spectral norm error:} 
Let $\CH_k=range(H_k)=span(h^{(1)},\ldots,h^{(k)})$ and let $\CH_{n-k}$ be the orthogonal
complement of $\CH_k$. For $x\in\RR^n$, let $x=\alpha y+\beta z$, where $y\in\CH_k,z\in\CH_{n-k}$
and $\alpha^2+\beta^2=1$. Then,
\begin{eqnarray*}
 \|A-H_kH_k^TA\|^2_2&=&\max_{x\in\RR^n:\|x\|=1}\|x^T(A-H_kH_k^TA)\|^2\\
 &=& \max_{y,z}\|(\alpha y^T+\beta z^T)(A-H_kH_k^TA)\|^2\\
 &\leq&\max_{y\in\CH_k:\|y\|=1}\|y^T(A-H_kH_k^TA)\|^2+\max_{z\in\CH_{n-k}:\|z\|=1}\|z^T(A-H_kH_k^TA)\|^2\\
 &=&\max_{z\in\CH_{n-k}:\|z\|=1}\|z^TA\|^2,
\end{eqnarray*}
since $\alpha,\beta\leq1$ and for any $y\in\CH_k,y^TH_kH_k^T=y^T$, so 
the first term is zero and 
for any $z\in\CH_{n-k},z^TH_kH_k^T=0$. Next,
\begin{eqnarray*}
 \|z^TA\|^2&=&\|z^TC\|^2+[\|z^TA\|^2-\|z^TC\|^2]\\
 &\leq& \sigma_{k+1}^2(C)+3\eps\|A\|_F^2\\
  &\leq& \sigma_{k+1}^2(A)+6\eps\|A\|_F^2\\
  &=&\|A-A_k\|_2^2+6\eps\|A\|_F^2.
\end{eqnarray*}
Since $|\|z^TA\|^2-\|z^TC\|^2|\leq3\eps\|A\|_F^2$ from Lemma~\ref{lemm:RQ}, 
$\max_{z\in\CH_{n-k}:\|z\|=1}\|z^TC\|^2=\sigma_{k+1}^2(C)$, and 
$|\sigma_i^2(C)-\sigma_i^2(A)|\leq \|AA^T-CC^T\|_2\leq3\eps\|A\|_F^2$.
\end{proof}
We  observe  that Theorem~\ref{theo:1}  is  similar to  the
results     developed     for     randomized     sampling     methods,
see~\cite{drineas2004clustering,drineas2006fast}.      One     notable
difference is that  to achieve the above result for  a given rank $k$,
the  method based  on column  norms for  randomized sampling  requires
$c=\Theta(k/\eps^2)$ columns to  be sampled to form $C$.   For a small
$\eps$, the number of columns $c$  will be quite large.  The error for
a given rank $k$ diminishes as $c$  increases, but when $k$ is large, a
large number of columns will be  required to get a good approximation.
For the coarsening method, the above  error bounds hold for any $k\leq
c$, and the  number of columns $c$ will depend  primarily on the given
data.  The  error will be  smaller if  the angles between  the columns
that are  combined are smaller.  The  number of columns is  related to
these  angles and this  in turn  depends on 
the  redundancy among columns   of the  given
matrix. As  future work it would be interesting to say
how many  distinct columns will be needed 
to ensure that the subspace spanned by the columns of $C$ 
is a good rank $k$ approximation to the range of $A$.
 
 The subspace iteration algorithm has been extensively studied, and 
 the most recent analyses of subspace iteration appeared 
 in~\cite{halko2011finding,gu2015subspace} and 
 \cite{saad2016analysis}.
The convergence of subspace iteration depends on the 
quality of  the initial subspace as an approximation to the 
dominant subspace, i.e., the subspace associated with the top singular
values.
The above theorem shows that the coarsening algorithm gives a reasonable 
approximation  to the subspace spanned by the top $k$ singular vectors.
This approximation is therefore a good starting point for 
the incremental SVD algorithm presented in section~\ref{sec:coarsen}.

\section{Applications}\label{sec:appl} 
Here, we  present a few  applications where the (multilevel) coarsening method
discussed in the previous section can be employed.
The matrices encountered  in these applications  are  typically large,
and  sparse, often representing graphs.

\paragraph{i.  Latent  Semantic Indexing  -} Latent  semantic indexing
(LSI)  is  a well-established  text  mining  technique for  processing
queries            in            a            collection            of
documents~\cite{deerwester1990indexing,landauer1998introduction,
  berry1995using,kokiopoulou2004polynomial}.   Given  a user's  query,
the  method is  used  to retrieve  a  set of  documents  from a  given
collection  that  are  most  relevant to  the  query.   The  truncated
SVD~\cite{berry1995using}                  and                 related
techniques~\cite{kokiopoulou2004polynomial} are  common tools  used in
LSI.   The  argument  exploited  is  that  a  low  rank  approximation
preserves the important underlying structure associated with terms and
documents,   and   removes   the   noise  or   variability   in   word
usage~\cite{drineas2006fast}.
Multilevel coarsening for LSI  was considered in \cite{hss:mlevel-08}.
Here, we revisit and expand this idea and show how hypergraph coarsening
can be employed in this application.

\paragraph{ii. Projective clustering -}
Several projective clustering methods such as Isomap~\cite{tenenbaum2000global},
Local Linear Embedding (LLE)~\cite{roweis2000nonlinear}, spectral 
clustering~\cite{ng2001spectral}, 
subspace clustering~\cite{parsons2004subspace,elhamifar2009sparse},
Laplacian eigenmaps~\cite{belkin2003laplacian} and others involve 
a partial eigen-decomposition or a partial SVD of graph Laplacians.
Various kernel based learning methods~\cite{muller2001introduction}
also require the (partial) SVD of large graph Laplacians.
 In most applications today, the number of data-points 
is large and computing  singular vectors (eigenvectors) is expensive 
in most cases. Graph coarsening is  an effective strategy to reduce
the number of  data-points in these applications, 
see~\cite{fang2010multilevel,oliveira2005multi} for a few illustrations.

\paragraph{iii. Eigengene analysis -}
Analyzing   gene expression  DNA microarray  data has  become an
important   tool   in  the   study   of   a  variety   of   biological
processes~\cite{alter2000singular,raychaudhuri2000principal,paschou2007intra}.
In  a microarray  dataset, we  have  $m$ genes  (from $m$  individuals
possibly from different populations) and  a series of $n$ arrays probe
genome-wide expression levels in $n$ different samples, possibly under
$n$ different experimental conditions.  The data is large with several
individuals and  gene expressions,  but is  known to  be of  low rank.
Hence,  it has  been  shown  that a  small  number  of eigengenes  and
eigenarrays (few singular  vectors) are sufficient to  capture most of
the     gene      expression     information~\cite{alter2000singular}.
Article~\cite{paschou2007intra}  showed  how column  subset  selection
(CSSP) can  be used for  selecting a  subset of gene  expressions that
describe the population well in terms of spectral information captured
by the reduction.  In the next section,  we will show how hypergraph coarsening can
be  adapted  to  choose  a  good  (small)  subset  of  genes  in  this
application.

\paragraph{iv. Multilabel Classification -}
The last application we consider is that of multilabel classification in machine 
learning~\cite{tsoumakas2008multi,ubaru17a}.
    In the multilabel classification problem, we are given a set of labeled training data
    $\{(x_i,y_i)\}_{i=1}^n$, where each $x_i\in\RR^p$ is an  input feature for a data  instance 
    which belongs to  one or more classes, and 
    $y_i\in\{0,1\}^d$  are vectors indicating the corresponding labels (classes) to which the data instances 
    belong. A vector $y_i$  has a one at the 
   $j$-th coordinate if the instance belongs to $j$-th class.
     We wish to learn a mapping (prediction rule) between the features and the labels, in order
 to be able to  predict a class label vector $y$ of a new data point $x$.
    Such multilabel classification problems occur in
    many domains such as  computer vision, text   mining,    and    bioinformatics
    \cite{trohidis2008multi,tai2012multilabel}, 
    and modern applications involve a large number of  labels. 
    
    A  common approach  to handle  classification problems  with many
    classes is to begin by reducing  the effective number of labels by
    means of so-called embedding-based approaches. The label dimension
    is  reduced by  projecting label  vectors onto  a low  dimensional
    space,   based   on  the   assumption   that   the  label   matrix
    $Y=[y_1,\ldots, y_n]$ has a low-rank.  The reduction  is achieved in
    different      ways,       for      example, by      using      SVD
    in~\cite{tai2012multilabel}  and column  subset  selection
    in~\cite{bi2013efficient}.   In  this  work,  we  demonstrate  how
    hypergraph  coarsening can  be employed  to reduce  the number  of
    classes, and yet achieve accurate learning and prediction.

    Article~\cite{savas2011clustered}  discusses a number of methods that
rely on clustering the data first in order to build a reduced dimension
 representation. It can be viewed as a top-down approach whereas coarsening 
is a bottom-up method.

\section{Numerical Experiments}\label{sec:expt}
This section describes a number of experiments to illustrate
the  performance of the different methods discussed. The latter 
part of the section focuses on  the performance of the coarsening
 method in the applications discussed above.

 \begin{figure}[tb]
\includegraphics[width=0.33\textwidth,trim={0.8cm 0cm 0.8cm 0cm}]
{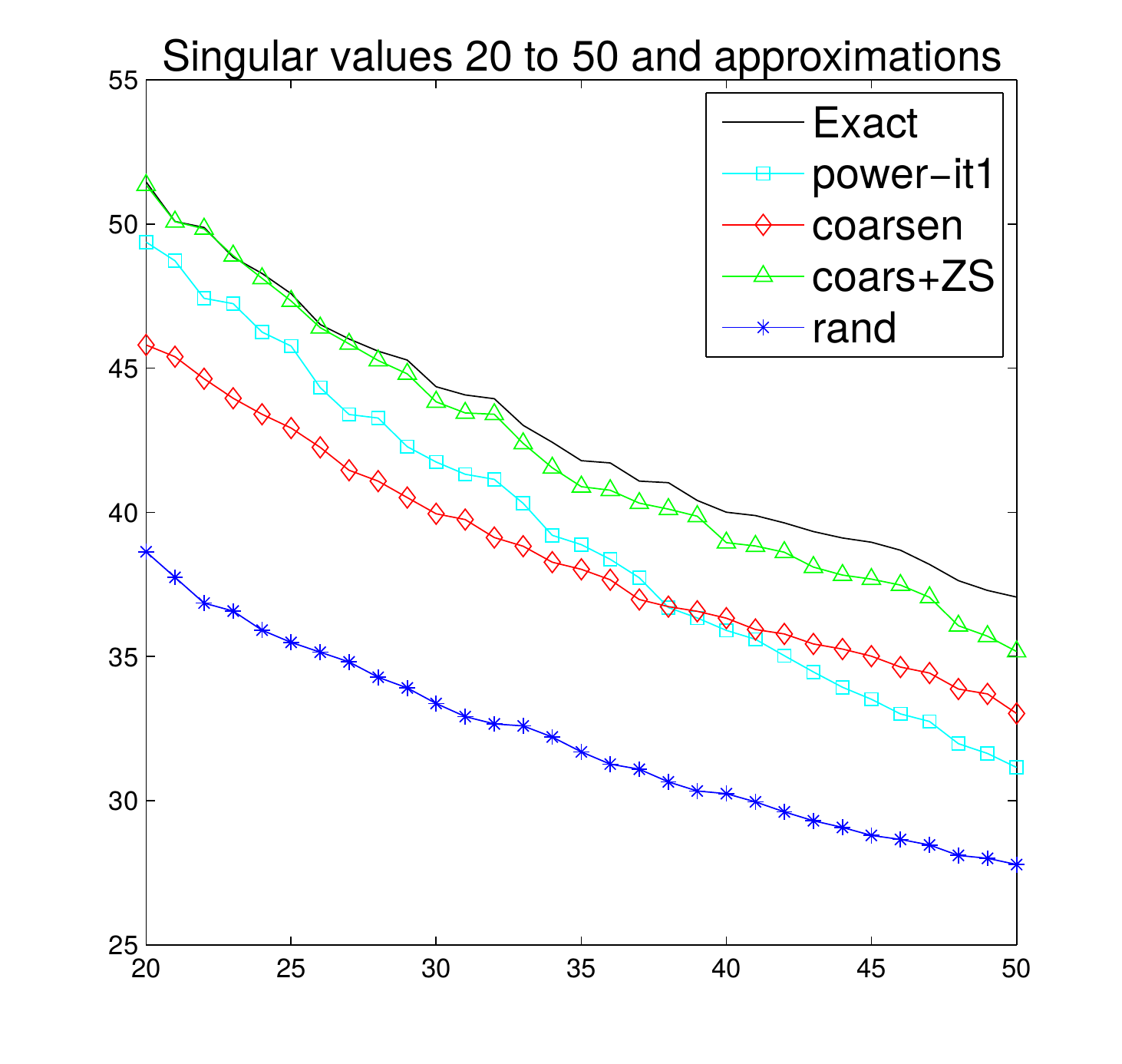}
\includegraphics[width=0.325\textwidth,trim={0.8cm 0cm 0.7cm 0cm}]
{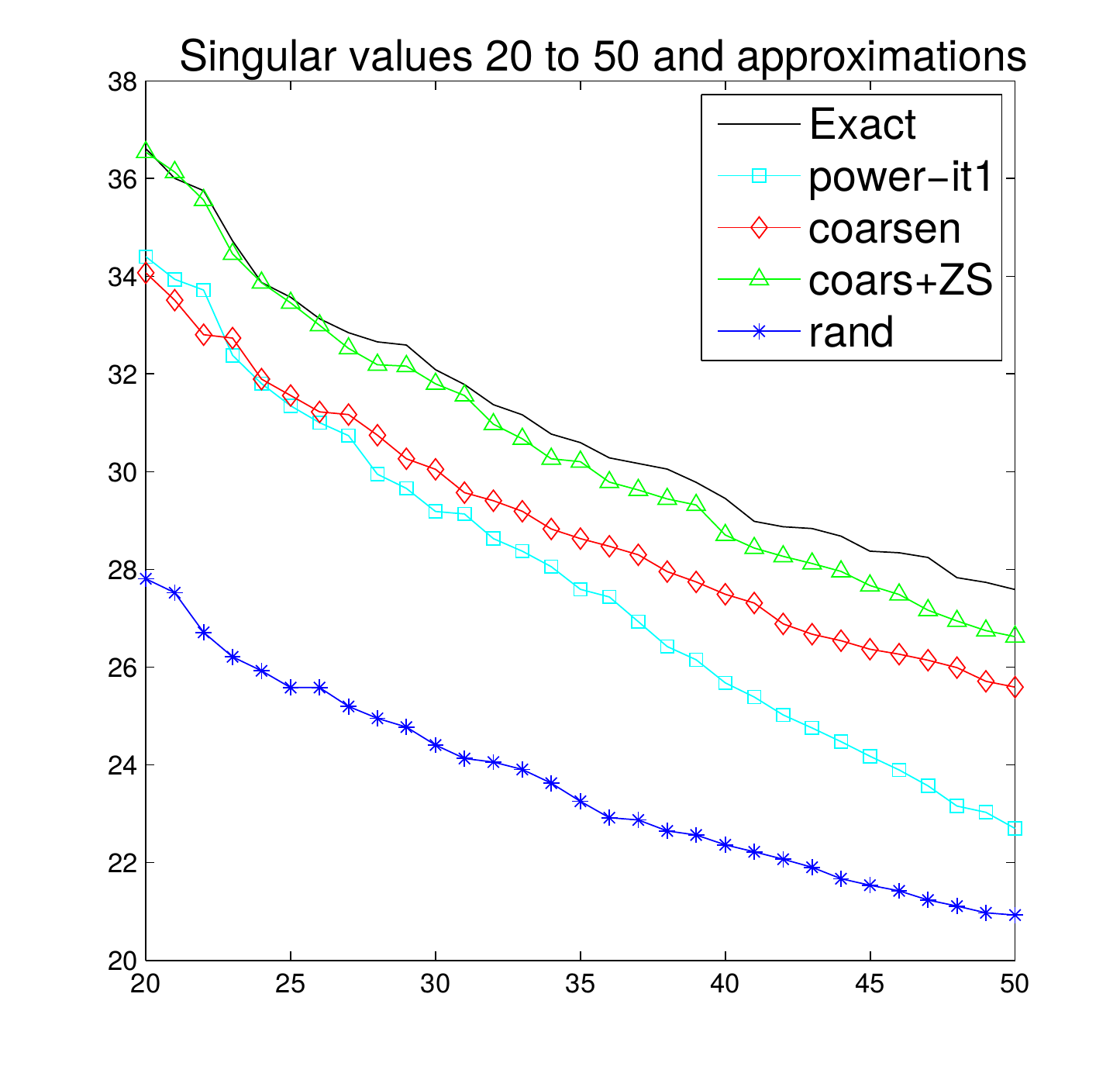} 
\includegraphics[width=0.325\textwidth,trim={0.8cm 0cm 0.8cm 0cm}]
{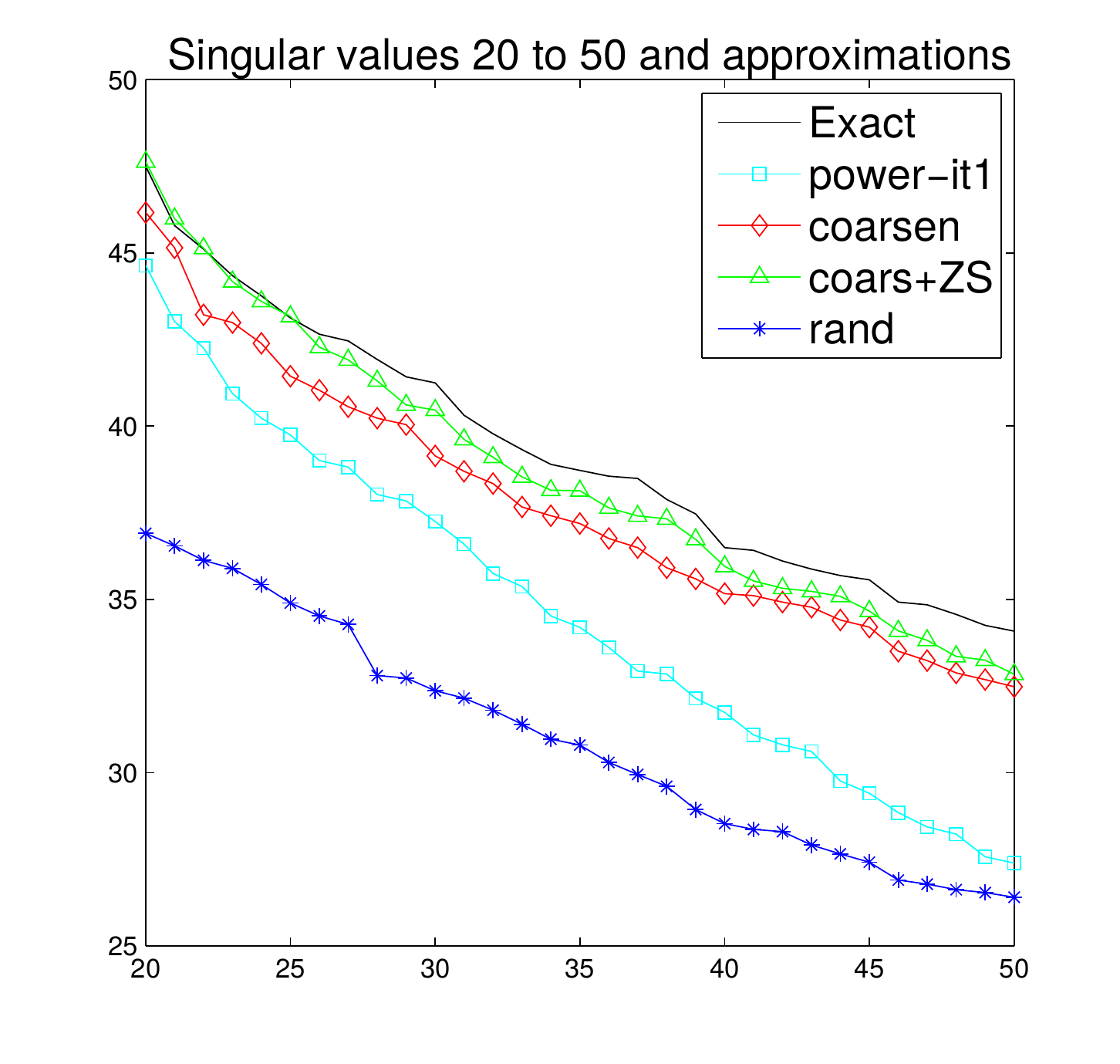} 
\vskip-0.1in
\caption{Results for the datasets CRANFIELD (left), MEDLINE (middle), and TIME (right).} 
\label{fig:coars_1}
\end{figure}

\subsection{SVD Comparisons}
In  the  first  set  of experiments,  we  start with  three small term-by-document
datasets and compare the  sampling, coarsening and combined methods
to compute the  SVD.  The tests
are with unweighted versions of the CRANFIELD dataset (1398 documents,
5204 terms),  MEDLINE dataset  (1033 documents,  8322 terms)  and TIME
dataset (425 documents, 13057 terms). These datasets
are used in the  latent
semantic indexing application examples. 

Figure \ref{fig:coars_1} illustrates the following experiment with the
three datasets.  Results from four  different methods are plotted. The
first solid curve (labeled `exact') shows the singular values numbered
20 to 50 of a matrix  $A$ computed using the \texttt{svds} function in
Matlab (the  results obtained by the  four methods for the  top twenty
singular values  were similar).  The diamond  curve labeled `coarsen',
shows the  singular values obtained  by one level of  coarsening using
Algorithm~\ref{alg:coarsening}  (details will  be  given later).   The
star  curve (labeled  `rand') shows  the singular  values obtained  by
random column  norm sampling method (columns  sampled with probability
$p_i=\|a_i\|^2_2/\|A\|_F^2$), with a sample size $c$ equal to the size
of coarsened matrix $C$ obtained with one level of coarsening. We note
that  the result  obtained  by  coarsening is  much  better than  that
obtained by random  sampling method, particularly for  large $k$. This
behavior was also  predicted by our theoretical  results.  However, we
know that the approximations obtained by either sampling or coarsening
cannot be highly  accurate.  In order to get improved  results, we can
invoke  the  incremental SVD  algorithms.   The  curve with  triangles
labeled `coars+ZS' shows  the singular values obtained  when the Zha Simon
algorithm was employed to improve the results obtained  by the coarsening
algorithm.   Here, we  consider  the singular  vectors  of the  coarse
matrix  and use  the  remaining part  of the  matrix  to update  these
singular  vectors and  singular  values.  We  have  also included  the
results      obtained     by      one      iteration     of      power
method~\cite{halko2011finding},  i.e.,  from  the SVD  of  the  matrix
$Y=(AA^T)A\Omega$,  where $\Omega$  is a  random Gaussian  matrix with
same number of columns as the  coarse matrix.  We see that the smaller
singular  values obtained  from the  coarsening algorithms  are better
than those obtained by the one-step power method.

 \begin{figure}[tb]
 \begin{center}
\includegraphics[height=0.35\textwidth,trim={1cm 0cm  1cm 0cm},clip]{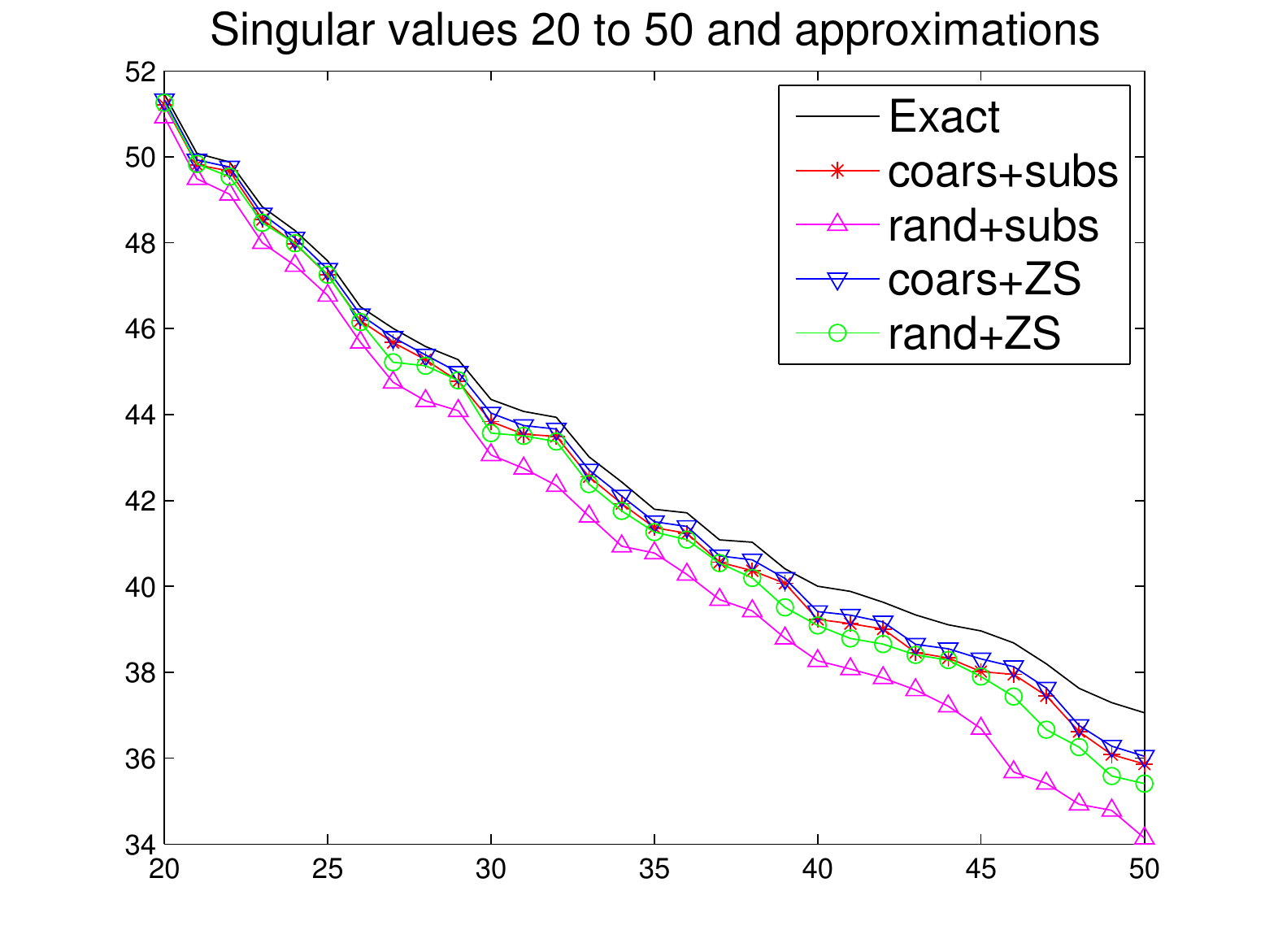}
\includegraphics[height=0.35\textwidth,trim={1cm 0cm  1cm 0cm},clip]{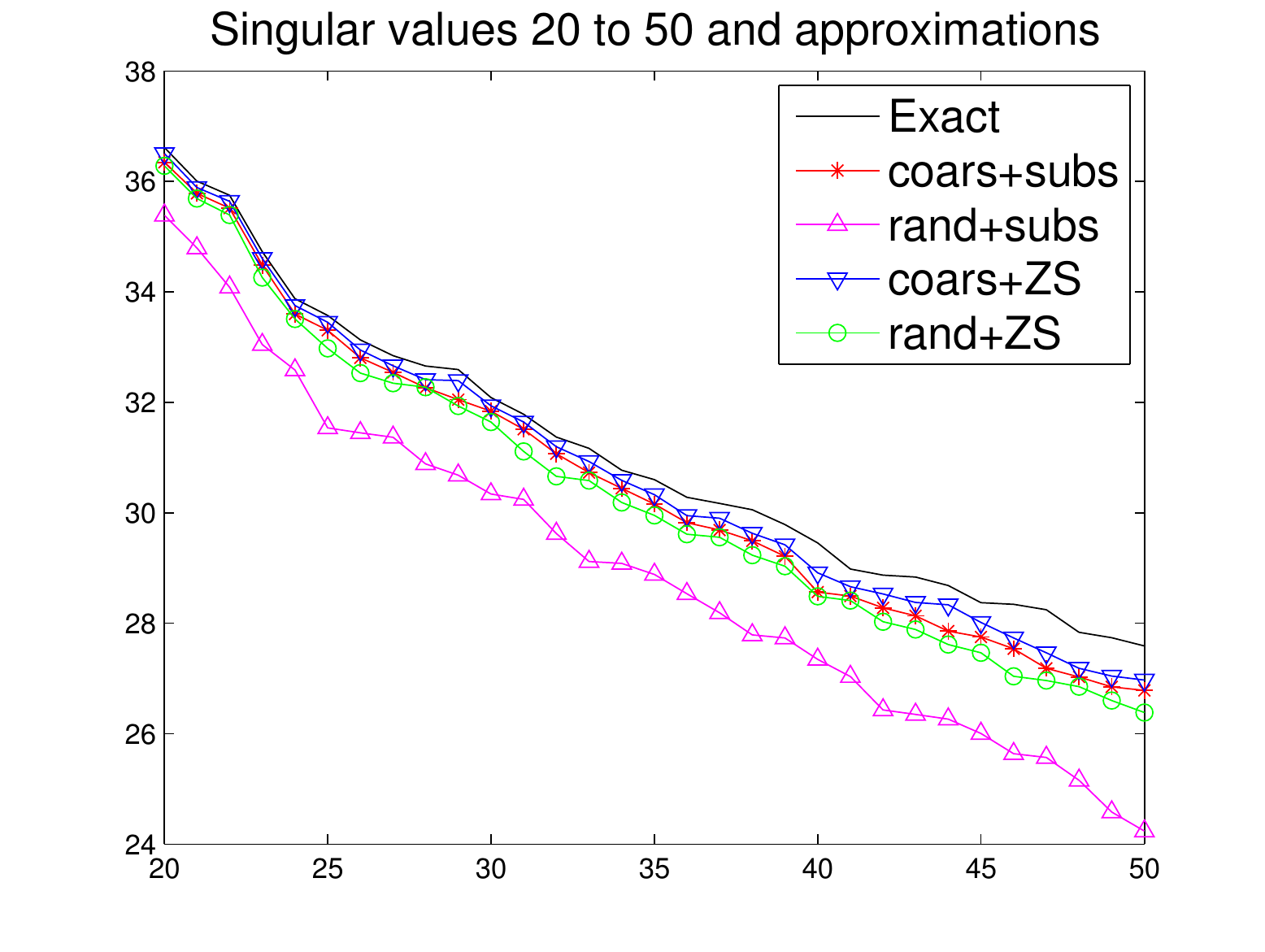} 
\vskip -0.1in
\caption{Second set of results for the CRANFIELD (left) and the MEDLINE  datasets (right).} 
\label{fig:coars_2}
 \end{center}
\end{figure}

As discussed in section~\ref{sec:coarsen}, a possible way to
improve the SVD results  obtained by a coarsening
or random sampling step is to resort to subspace iteration 
or use the SVD update algorithms as in the first experiment.
Figure~\ref{fig:coars_2}  illustrates such results 
with incremental SVD algorithms for the CRANFIELD (left) and the 
MEDLINE (right) datasets. We have not reported the results for the TIME dataset since 
it is hard to distinguish the results obtained by different algorithms for this case.
First, subspace iteration is
performed using the matrix $A$ and the singular vectors obtained from coarsening or 
column norm sampling. 
The curve `coars+subs' (star) corresponds to the singular values obtained when subspace iteration  
was used to improve the SVD obtained by coarsening. Similarly, for the
curve labeled `rand+subs' (triangle up), subspace iteration was used with the singular vectors 
obtained from randomized sampling. We have included the results when the SVD update algorithm 
was used to improve the SVD obtained by coarsening (`coars+ZS') and column norm 
sampling (`rand+ZS'),
respectively.
These plots show that both the SVD update algorithm and subspace iteration
improve the accuracy 
of the SVD significantly.

\begin{table*}[tb!]
\begin{center}
\caption{Low rank approximation: Coarsening, col. norm sampling, and rand+coarsening.
Error1 $=\|A-H_kH_k^TA\|_F$; Error2$=\frac{1}{k}\sum_{i=1}^k\frac{|\hat \sigma_i-\sigma_i|}{\sigma_i}$.}
\label{tab:svdcomp}
{\small
\vskip -0.1in
\begin{tabular}{|l|r|c|r|c|c|c|c|c|c|c|c|c|c|}
\hline
Dataset& \centering{$n$} &$k$& \centering{$c$}&\multicolumn{3}{c|}{Coarsening}&\multicolumn{3}{c|}{Col. norm Sampling}
&\multicolumn{3}{c|}{Rand+Coarsening}\\\cline{5-13}
&&&&Err1&Err2&{time}&Err1&Err2&{time}&Err1&Err2&{time}\\
\hline
Kohonen& 4470& 50 & 1256 & 86.26 & 0.366&1.52s& 93.07 &  0.434 &0.30s & 93.47& 0.566&0.43s\\
aft01 (n)& 8205& 50& 1040 & 913.3& 0.299&2.26s& 1006.2 & 0.614 &0.81s& 985.3 & 0.598&0.63s \\
FA & 10617& 30& 1504 & 27.79& 0.131&2.06s& 28.63 & 0.410 &1.83s& 28.38 & 0.288 & 0.43s\\
chipcool0 & 20082& 30& 2533 & 6.091& 0.313&5.69s& 6.199 & 0.360 &5.92s& 6.183 & 0.301 &1.27s\\
{scfxm1-2b} & 33047 & 25& 2066 & 2256.8 &0.202&2.8s& 2328.8 & 0.263&5.38s& 2327.5 &  0.153&0.57s\\
\hline
thermechTC & 102158& 30& 6286 & 2.063 &0.214&7.8s& 2.079 & 0.279&164s& 2.076 &  0.269&2.51s\\
{HTC9129(n)} & 226340& 25& 14147 & 212.8&0.074&247.6s& 226.34 & 0.102&481.7s& 234.08 & 0.1512&57.8s\\
{ASIC-680(n)}& 682712& 25& 21335& 2913& 0.239&57.4s& 2925 & 0.304 &17.3m& 2931 &  0.448 &13.2s\\
\hline
Webbase1M & 1000005& 25& 15625 & 3502.2 &--&15.3m& 3564.5 & --  &17.8m& 3551.7 &  -- &5.1m\\
{Netherland} & 2216688& 25& 34636 &1.53e4 &--&51.8s& 2.05e4& --  &584.3s& 1.58e4 &  -- &15.3s\\
{adaptive}& 6815744& 25& 106496 &2850.2&--&297.4s&  3256.8& --  &52.6m& 2984.7 &  -- &98.3s\\
{delaunay-23}&8388608& 25& 131072 & 3145.4&--&13.2m& 3904.2 & --  &38.2m& 3446.1 &  -- &3.8m\\
\hline
\end{tabular}
}
\end{center}
\end{table*}

 \begin{figure}[tb]
 \begin{center}
\includegraphics[height=0.31\textwidth,trim={0.5cm 0cm  1cm 0cm},clip]{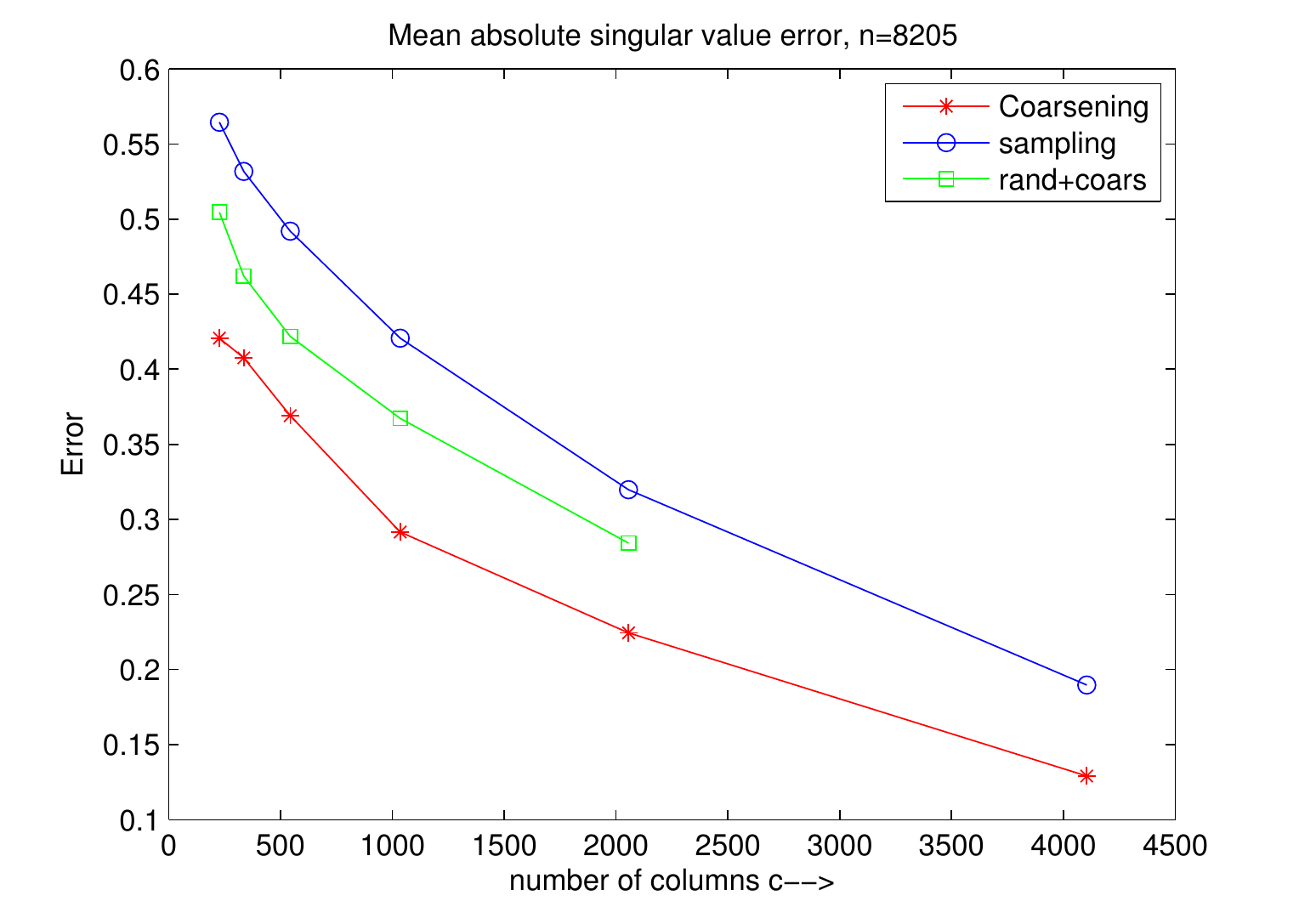}
\includegraphics[height=0.31\textwidth,trim={0.4cm 0cm  1cm 0cm},clip]{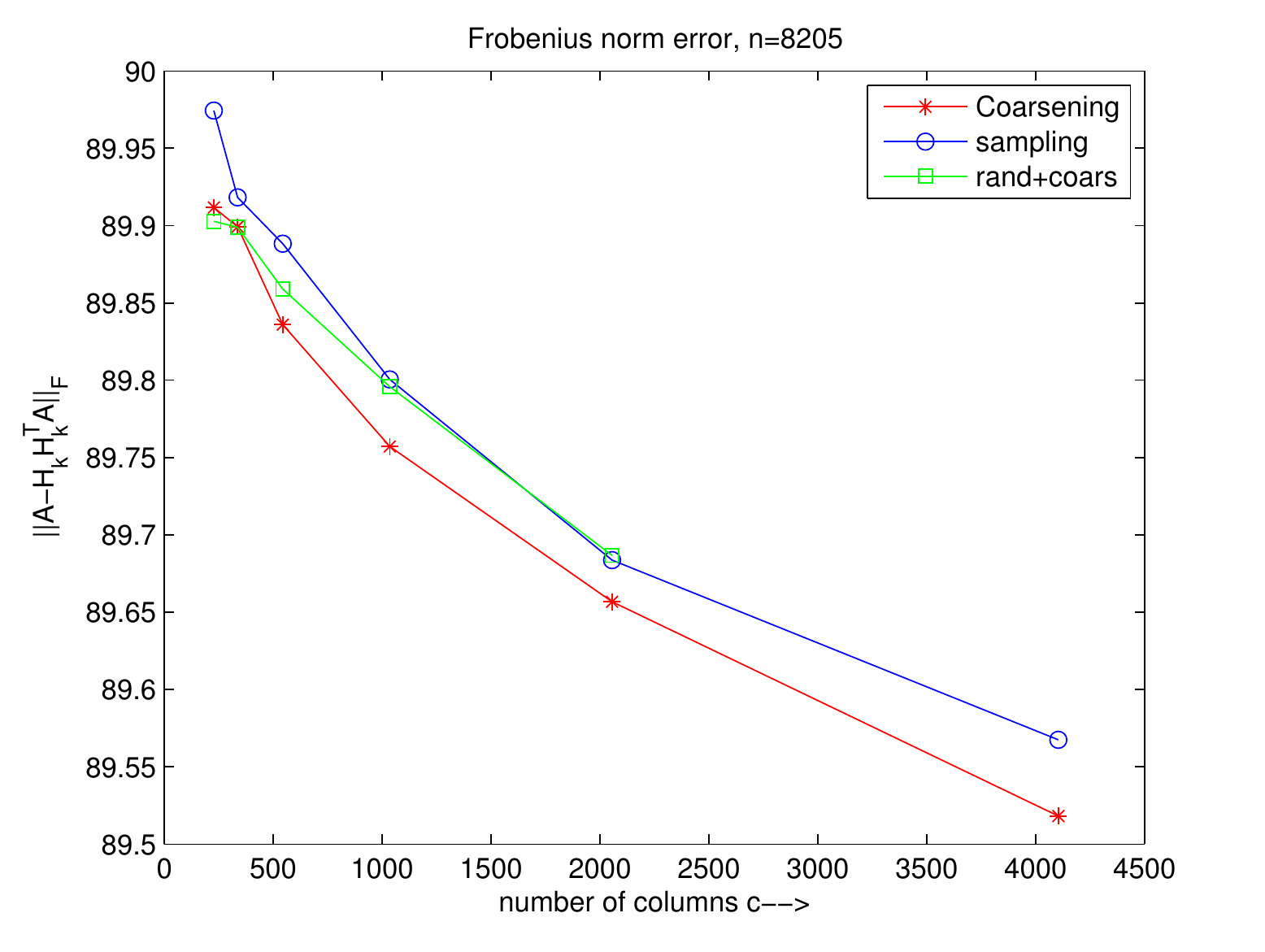} 
\vskip -0.1in
\caption{Mean absolute singular value errors 
$\frac{1}{k}\sum_{i=1}^k\frac{|\hat \sigma_i-\sigma_i|}{\sigma_i}$ (Left) and 
Frobenius norm errors $\|A-H_kH_k^TA\|_F$ (right) for the three methods 
for \texttt{aft01} dataset ($k=30$).} 
\label{fig:coars_3}
 \end{center}
\end{figure}

Next, we compare the {computational times and} performances of our coarsening method and the column norm sampling
method for computing the low rank approximation 
of matrices. We also consider the combined method of uniform downsampling followed by
coarsening discussed in the
introduction and in section~\ref{sec:coarsen}.
Table~\ref{tab:svdcomp} shows comparison results between the three methods, namely,
Coarsening, column norm sampling, and uniform random sampling+coarsening for low rank approximation 
of matrices from various applications.
All matrices were obtained from 
the SuiteSparse matrix collection: 
\url{https://sparse.tamu.edu/}~\cite{davis2011university} and are sparse.
The errors reported are the Frobenius norm error $=\|A-H_kH_k^TA\|_F$ in computing the 
rank $k$ approximation ($k$ listed in third column).
For the smaller matrices, we also report the average
absolute normalized error in the singular values
$=\frac{1}{k}\sum_{i\le k}\frac{|\hat \sigma_i-\sigma_i|}{\sigma_i}$
for the rank $k$, which is related to our result in Lemma~\ref{lemm:RQ}
and spectral approximation which is of interest. {We also report the computational time in seconds (s) or minutes (m)  for the three methods (using \texttt{cputime} function in matlab). All experiments were
conducted on an Intel i5-4590 CPU @3.30GHz machine.}
The size of the input matrix and the number of columns in the 
coarsened/subsampled matrix are listed in the second and fourth columns, respectively.
{ Some of the matrices (indicated by (n)) which had large Frobenius norm were normalized by $\sqrt{\|X\|_F}$.
For smaller matrices, the exact singular values $\sigma_i$'s are available in the database. For matrices with $n>10^5$,  these were computed using `svds' function.
For larger matrices with $n>10^6$, the exact singular values cannot be computed, hence
we were unable to report Error2 for these matrices. }

{For the combined method (random sampling+coarsening) we proceed as follows:  First, half of the columns ($n/2$) are randomly sampled (uniform downsampling),
and then a multilevel coarsening is performed with one level less than the pure coarsening method
reported in the previous column. 
We note that, coarsening clearly yields  better results (lower errors) than the column norm sampling method. 
The combined method of random sampling+coarsening works 
well and performs better than norm sampling in most cases. Coarsening is relatively  expensive for smaller matrices.
However, we note that for larger (and sparse) matrices, coarsening becomes less expensive than column norm sampling, which requires computing the norms of $n$ columns and then sample based on the ratio of the norms.
For really large matrices, the `rand+coarse' method seems to be a good choice due to its computational cost.
The coarsening methods scale easily to even larger matrices. However, estimating the error $\|A-H_kH_k^TA\|_F$ becomes impractical.}
 
For further illustration, figure~\ref{fig:coars_3} plots the two errors $\|A-H_kH_k^TA\|_F$  and
$\frac{1}{k}\sum_{i\le k}\frac{|\hat \sigma_i-\sigma_i|}{\sigma_i}$ with $k=30$
for the three methods for the \texttt{aft01} dataset 
when different levels of coarsening were used, i.e., 
the number of columns sampled/coarsened were increased. 
For `rand+coars', we do not have errors for $c=n/2$ since this was obtained uniform sampling.
For a  smaller number of columns, i.e., more levels in coarsening, 
the Frobenius norm error for 
rand+coarsen  approaches that of pure coarsening.

In all the above experiments, we have used maximum matching for coarsening.
That is, in Algorithm~\ref{alg:coarsening}, we match the current column with 
the column with which it has maximum inner product. The choice of $\epsilon$, 
the parameter that decides the angle for matching does not seem to affect the errors 
directly. 
If we choose smaller $\epsilon$, we will have a larger coarse matrix $C$ (fewer columns
are combined) and the error will  be small. If we choose a larger $\epsilon$, more 
columns are combined and the results are typically equivalent to just simply using
maximum matching ignoring the angle constraint. Thus, in general, 
the performance of the coarsening technique
depends on the intrinsic properties of the matrix under consideration. 
If we have more columns that are close to each other,
i.e., that make a small angle between each other,
the coarsening technique will combine more columns, we can 
choose a smaller $\eps$ and yet 
obtain good results. If the matrix is very sparse 
or if the columns make large angles between each other, 
 coarsening might yield a  coarse matrix that has nearly as many 
columns as the original matrix since 
it will not be able to match many columns.

\subsection{Column Subset Selection}
In the following experiment, we compare the performance of the coarsening 
method against the leverage score
sampling method for column subset selection (CSSP). We report results
for the term-by-document datasets 
(used in  the first set of experiments) and for a few sparse matrices 
from the SuiteSparse matrix collection.

\begin{table*}[tb!]
\caption{CSSP: Coarsening versus leverage score sampling.}
\label{tab:table1}
\begin{center}
{\small
\begin{tabular}{|l|r|c|r|c|c|c|c|c|}
\hline
Dataset&Size&Rank $k$& $c$&\multicolumn{3}{c|}{Coarsening}&\multicolumn{2}{c|}{levSamp}\\\cline{5-9}
&&&&levels&error&{time}&error&{time}\\
\hline
TIME&425 &25&107&2&411.71&6.4s&412.77&0.52s\\
&&50&107&2&371.35&6.4s&372.66&1.08s\\
&&50&54&3&389.69&6.8s&391.91&1.08s\\
\hline
MED&1033
&50&65&4&384.91&5.3s&376.23&1.1s\\
&&100&130&3&341.51&5.3s&339.01&2.2s\\
\hline
Kohonen&4470&25&981&3&31.89&2.4s&36.36&1.4s\\
Erdos992&6100&50&924&3&100.9&0.85s&99.29&3.21s\\
FA&10617&50&2051&3&26.33&2.53s&28.37&3.72s\\
chipcool0&20082&100&1405&4&6.05& 7.02s&6.14&35.69s\\
{scfxm1-2b} & 33047 & 100& 2066&4&2195.4& 4.67s&2149.4&14.46s\\
{thermomechTC} & 102158& 100&  6385&4&2.03& 12.7s&2.01&166.1s\\
{ASIC-680 (n)}& 682712& 100&  42670&4&304.9& 57.3s&322.1&179.8s\\
\hline
\end{tabular}
}
\end{center}
\end{table*}

{Table~\ref{tab:table1} presents the performances and computational times of the two methods.}
The errors reported are the
Frobenius norm errors $\|A-P_CA\|_F$, where
$P_C$ is the projector onto $span(C)$,
and $C$ is the coarsened/sampled matrix
which is computed by the multilevel coarsening method 
or using  leverage score  sampling of  $A$ with  the top  $k$ singular
vectors as  reported in  the second  column. Note  that $P_C=H_kH_k^T$
from Theorem~\ref{theo:1}, with  $k=c$.  {For larger matrices, we computed the projector as
 $P_C=H_{2k}H_{2k}^T$ since it was not practical to computed the exact projector for such large matrices (even after coarsening).} 
The number of  columns $c$ in
each test is reported  in the third column which is  the same for both
methods. Recall that for CSSP,  the coarsening and sampling algorithms
do not  perform a post-scaling of  the columns that are  selected. 
The multilevel coarsening  method performs quite well, yielding results
that are comparable with those of 
leverage  score sampling. 
{Recall that the 
standard leverage score  sampling requires the computation  of the $k$
top singular vectors which  is inexpensive for small sparse matrices, but can be substantially more expensive than
coarsening for large sparse matrices, especially when $k$ is large.
Hence, we note that for the large matrices, coarsening is less expensive than leverage score sampling.
Again, the coarsening method easily scales to even larger matrices. However, computing the errors (and leverage scores)
become inviable.}

\subsection{Graph Sparsification}
The  next  experiment  illustrates  how coarsening  can  be  used  for
(spectral) graph sparsification.  We  again compare the performance of
the   coarsening    approach   to   the   leverage    score   sampling
method~\cite{kapralov2014single}  for  graph spectral  sparsification.
Recall  that spectral  sparsification  amounts to  computing a  sparse
graph $\tilde  G$ that approximates  the original graph $G$  such that
the singular  values of the graph  Laplacian $\tilde K$ of  $\tilde G$
are close to those of $K$, Laplacian of $G$.

\begin{table*}[tb!]
\caption{Graph Sparsification: Coarsening versus leverage score sampling.
Error$=\frac{1}{r}\sum_{i=1}^r\frac{|\sigma_i(\tilde K)-\sigma_i(K)|}{\sigma_i(K)}$}
\label{tab:tabspar}
\begin{center}
{\small
\begin{tabular}{|l|r|r|c|c|c|c|c|c|}
\hline
Dataset&$m$ & $r$&$\frac{nnz(\tilde K)}{nnz(K)}$&\multicolumn{3}{c|}{Coarsening}&\multicolumn{2}{c|}{levSamp}\\\cline{5-9}
&&&&levels&error&{time}&error&{time}\\
\hline
sprand& 1954 &489&0.27&2&0.545&0.84s&0.578&0.23s\\
          & {17670}& 4418&0.257&2&0.283&9.05s&0.371&8.27s\\
 	 & {49478} &12370&0.251&2& 0.102&41.9s&0.243&33.6s\\
	 & {110660} &27665&0.253&2& 0.089&118.9s&0.117&244.3s\\
	  & {110660} &13888  &0.129&3& 0.144&130.7s&0.308&242.1s\\
	   &{ 249041}&31131 &0.128&4& 0.104&426.1s&0.165&731.0s\\
  \hline
G1& 19176&4794&0.265&2&0.151&4.7s&0.221&17.1s\\
{comsol}& 49585&6337&0.281&2&0.157&5.4s&0.154&21.3s\\
{heart2}& 171254&21407&0.126&3&0.265& 121.5s&0.241&126.3s\\
{heart1}& 347853&21335&0.126&3&0.166&409.7s&0.208&515.8s\\
\hline
\end{tabular}
}
\end{center}
\end{table*}

Table~\ref{tab:tabspar} lists the errors obtained {and the computational costs} when the coarsening
and  the leverage  score sampling  approaches were  used to  compute a
sparse graph $\tilde  G$ for different sparse random graphs  and a few
matrices related  to graphs  from the  SuiteSparse collection.   Given a
graph $G$, we  can form a vertex edge incidence  matrix $B$, such that
the graph Laplacian of $G$
 is $K=B^TB$.  Then, sampling/coarsening  the rows of $B$ to
get   $\tilde   B$   gives   us  a   sparse   graph   with   Laplacian
$\tilde{K}=\tilde{B}^T\tilde{B}$. The  type of graph or  the names are
given in the first  column of the table and the number  of rows $m$ in
corresponding vertex edge incidence matrix  $B$ is given in the second
column.  The  number of rows  $r$ in the  coarse matrix $\tilde  B$ is
listed in the third column.  The  ratios of sparsity in $\tilde K$ and
$K$  are also  given.  This  ratio indicates  the  amount of  sparsity
achieved by sampling/coarsening. Since, we have the same number of rows in
the coarsened  and sampled matrix $\tilde  B$, this ratio will  be the
same  for both  methods.  The  error reported  is the  normalized mean
absolute  error  in  the  singular  values  of  $K$  and  $\tilde  K$,
Error$=\frac{1}{r}\sum_{i=1}^r\frac{|\sigma_i(\tilde
  K)-\sigma_i(K)|}{\sigma_i(K)}$, 
which tells us how close the sparser
matrix  $\tilde K$  is to  $K$ spectrally  (related to  the result  in
Lemma~\ref{lemm:RQ}).   {We  see that  in  most cases, the  coarsening
approach yields a smaller error  than with leverage score sampling.
This is particularly true when the input graph is relatively dense and its pattern has certain structure}.

\subsection{Applications} In this section, we illustrate the performance of the coarsening 
technique in the various applications introduced in the previous section~\ref{sec:appl}.

\begin{figure}[tb]
\begin{center}

\includegraphics[height=0.3\textwidth]{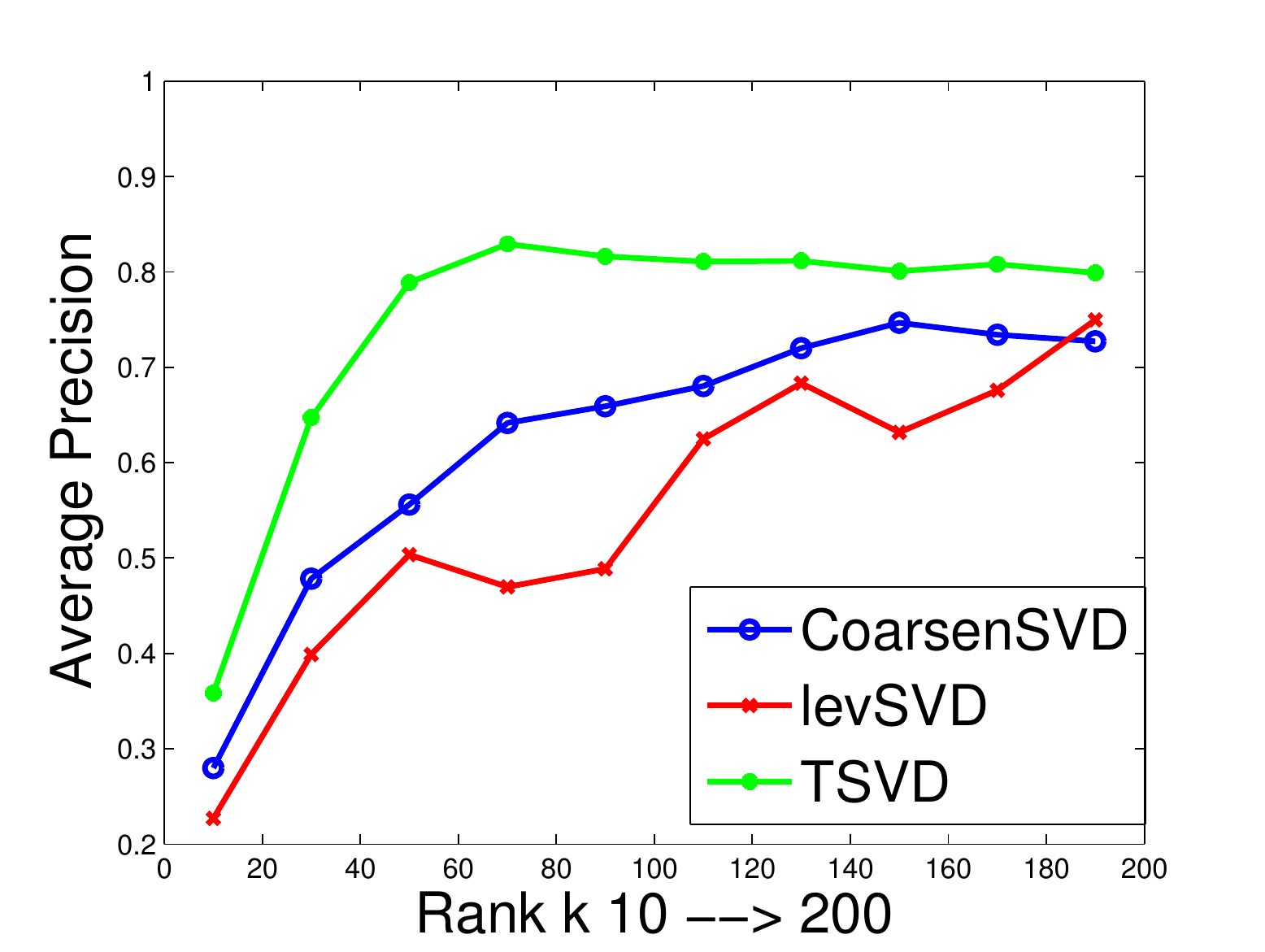}
\includegraphics[height=0.3\textwidth]{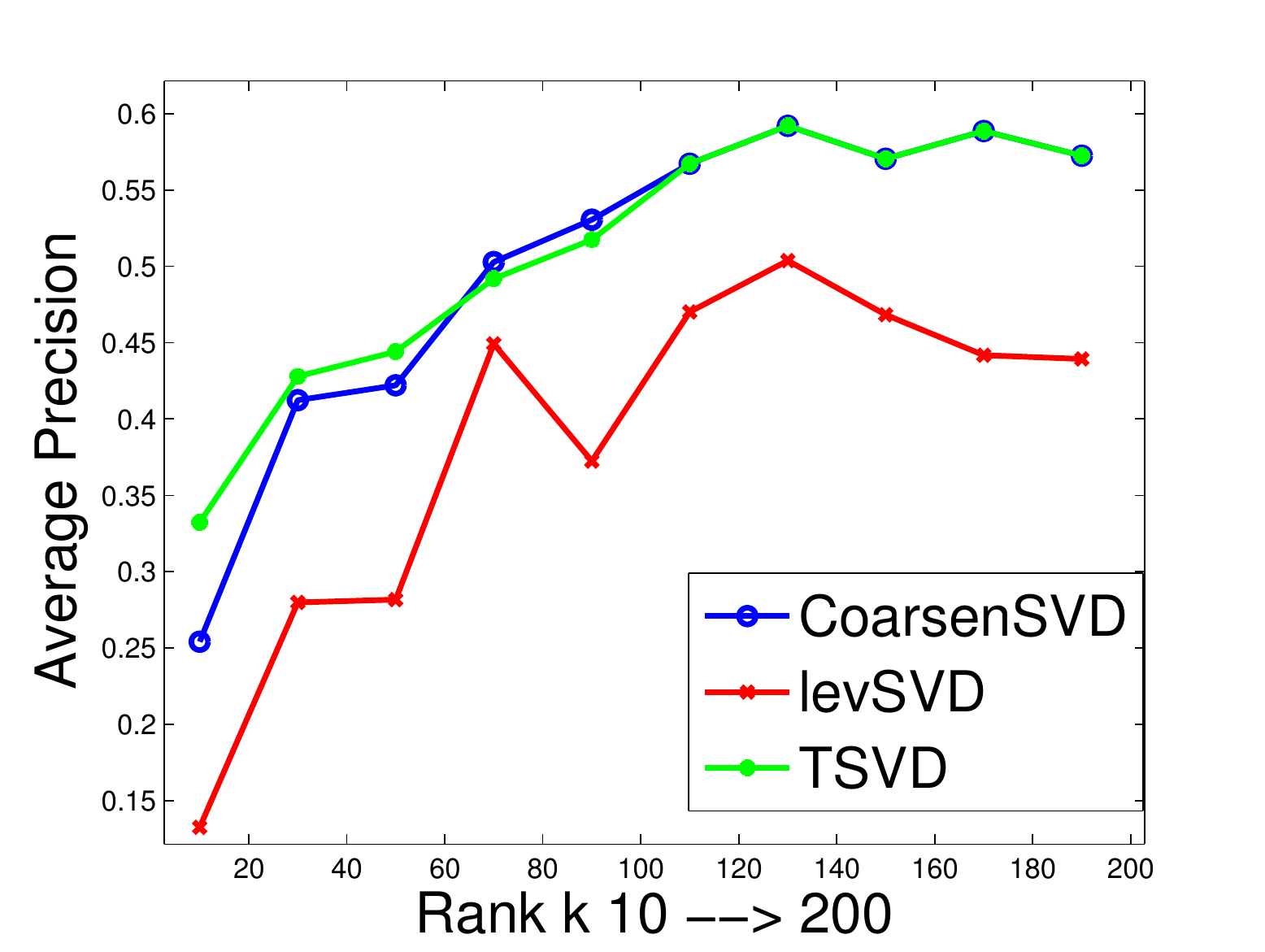} 
\vskip -0.1in
\caption{LSI results for the MEDLINE dataset on left and TIME dataset on the right.} 
\label{fig:LSI_1}
\end{center}
\end{figure}

\subsubsection{Latent Semantic Indexing}
The first application we 
consider  is  Latent Semantic Indexing 
(LSI)~\cite{deerwester1990indexing,landauer1998introduction}.
In LSI, we have a term-document matrix
$A\in \RR^{m\times n}$,
representing 
$m$ documents and $n$ terms that  frequently occur in the documents,
where $A_{ij}$ is the frequency of the $j$-th term in the $i$-th document.  
A query is an $n$-vector $q\in\RR^n$, normalized to $1$, where the
$j$-th component of a query vector is interpreted as the frequency with which the
$j$-th term occurs in a  topic. Typically, the number of topics 
to which the documents are related 
is 
smaller than
 the number of unique terms $n$.  Hence, 
finding a set of $k$ topics that best describe the collection of  documents for
a given $k$, corresponds 
to keeping only the top $k$ singular vectors of $A$, and obtaining a rank $k$ approximation. 
 The  truncated SVD and related methods
are often
 used  in LSI applications. The argument is that a low rank approximation
captures the important underlying intrinsic semantic 
 associated with terms and documents, and removes
the noise or variability in word usage~\cite{landauer1998introduction}.
In this experiment, we employ the Coarsen SVD and leverage score sampling SVD algorithms  
to perform  information retrieval techniques by Latent Semantic Indexing (LSI)~\cite{hss:mlevel-08}. 

Given a term-by-document data $A\in\RR^{m\times n}$, we normalize the data 
using TF-IDF (term frequency-inverse document frequency) scaling.
We also normalize the columns to unit vectors.
Query matching is the process of finding the documents most
relevant to a given query $q\in\RR^m$.

Figure~\ref{fig:LSI_1} plots the average precision against the dimension/rank $k$ for 
MEDLINE  and TIME datasets.
When the term-document matrix $A$
is large, the computation of the SVD factorization can be expensive for large
ranks $k$. The multi-level 
techniques  will find a smaller set of document vectors, denoted by$
A_r\in\RR^{m\times n_r}$, to represent $A$ ($n_r< n$). For leverage score sampling, we sample $A_r$
using leverage scores with $k$ equal to the rank shown on the $x$  axis.
Just like in  the standard LSI, we compute the truncated SVD of
$A_r=U_d\Sigma_dV_d^T$, where $d$ is the rank (dimension) chosen. Now the
reduced representation of $A$ is $\hat{A}=\Sigma_d^{-1}U_d^TA$.
Each query $q$ is transformed to a reduced representation  $\hat{q}=\Sigma_d^{-1}U_d^Tq$. 
The similarity of $q$ and $a_i$
are measured by the cosine distance between $\hat{q}$ and $\hat{a}$
for $i=1,\ldots,n$. 
This example clearly illustrates the advantage of the coarsening method over randomized sampling and
leverage scores. The multilevel coarsening method performs  better than the
sampling method in this application and in some cases it performs as well as the truncated SVD method.
 Multilevel coarsening algorithms  for LSI applications, have been discussed in~\cite{hss:mlevel-08} 
 where additional details can be found.

   \begin{figure}[tb]
\includegraphics[height=0.27\textwidth]{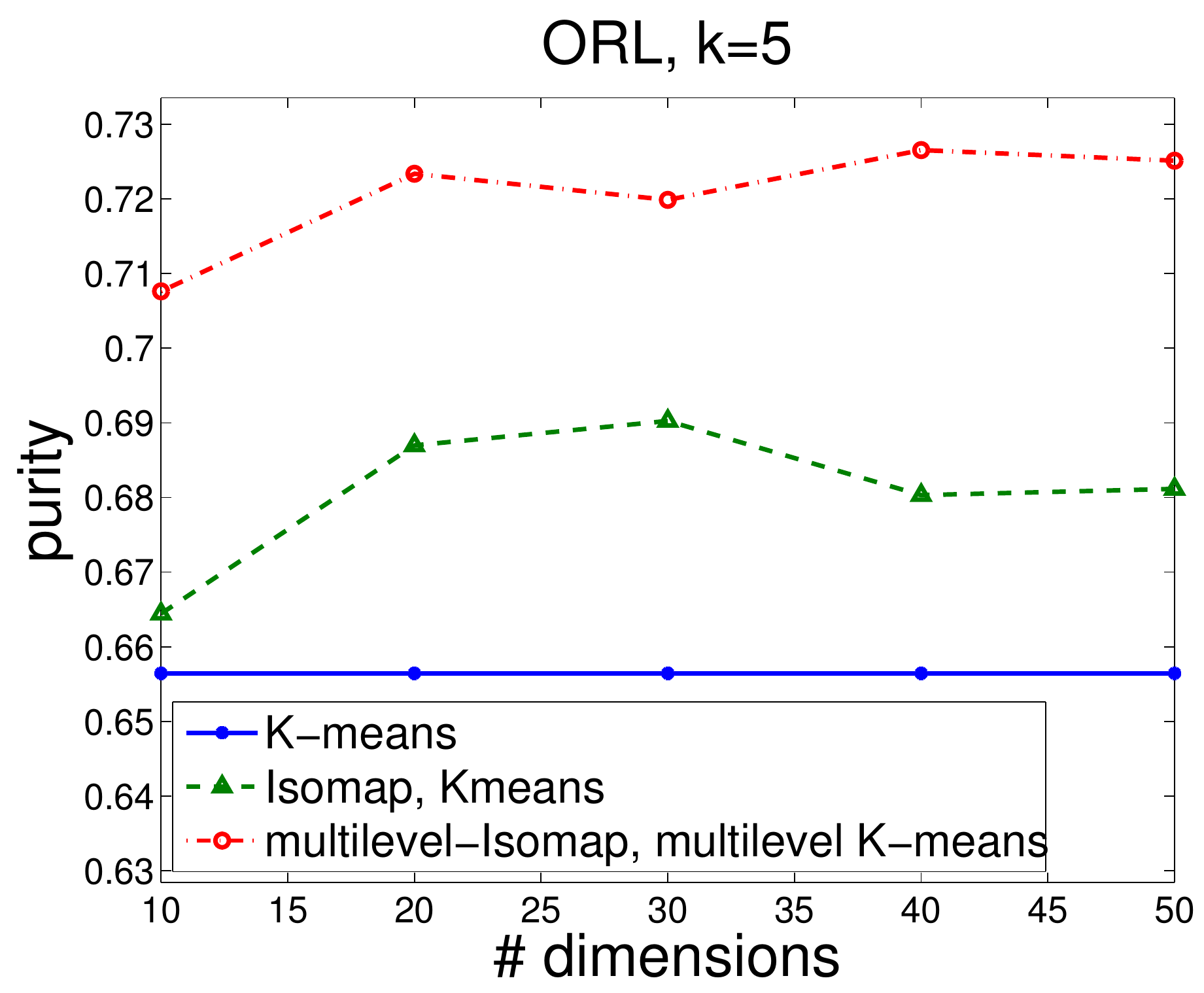}
\includegraphics[height=0.27\textwidth]{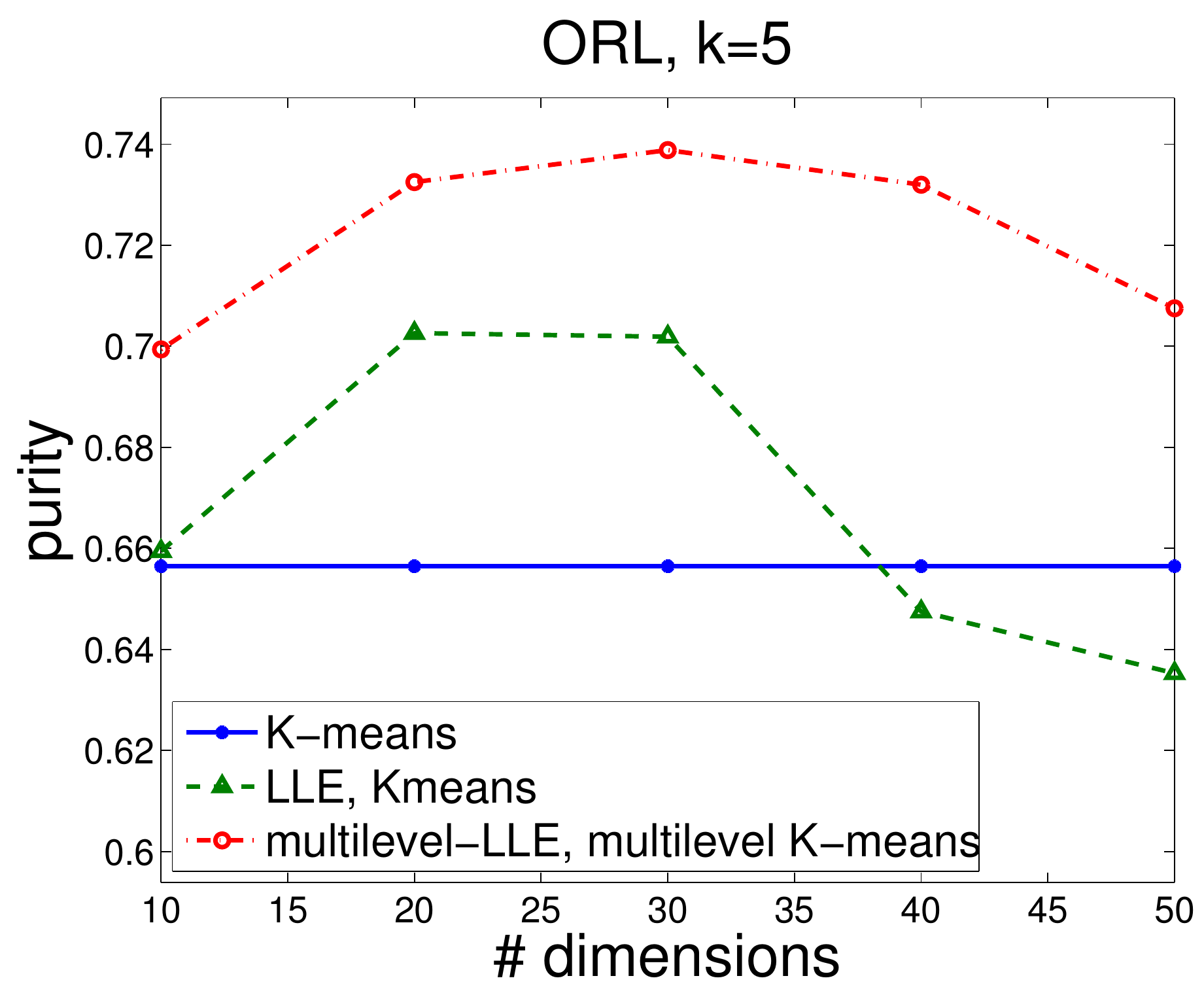} 
\includegraphics[height=0.27\textwidth]{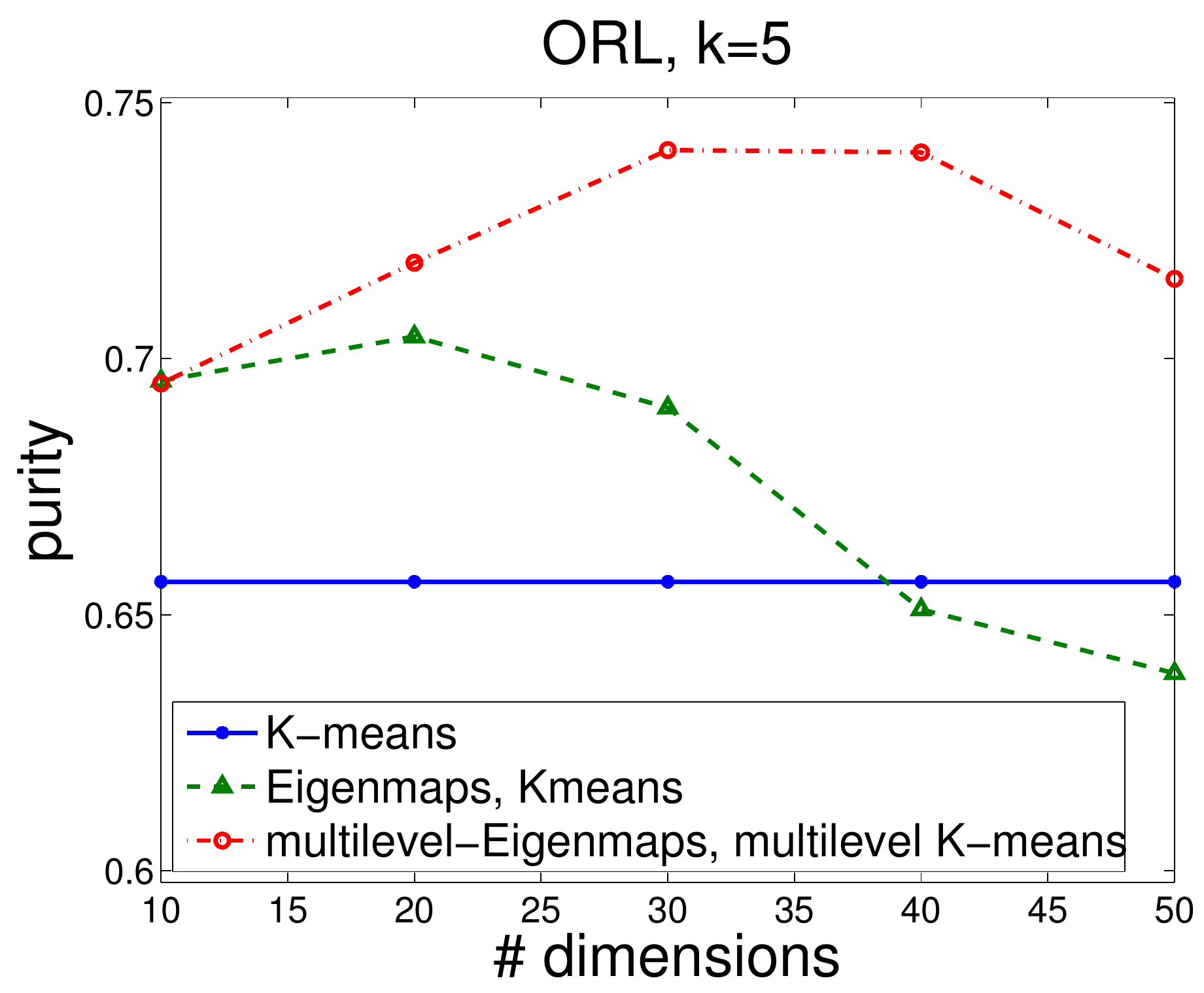} 
\includegraphics[height=0.27\textwidth]{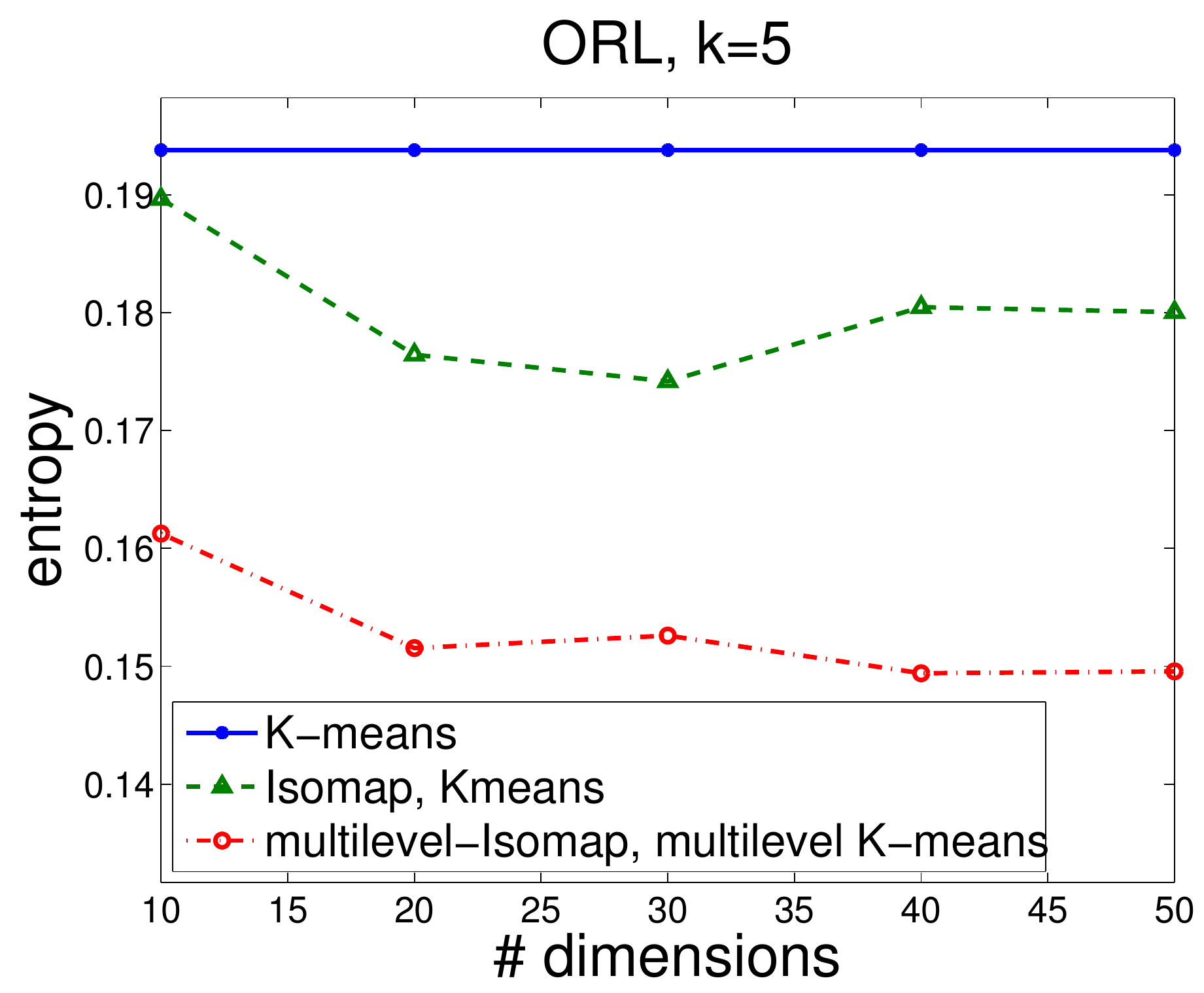}
\includegraphics[height=0.27\textwidth]{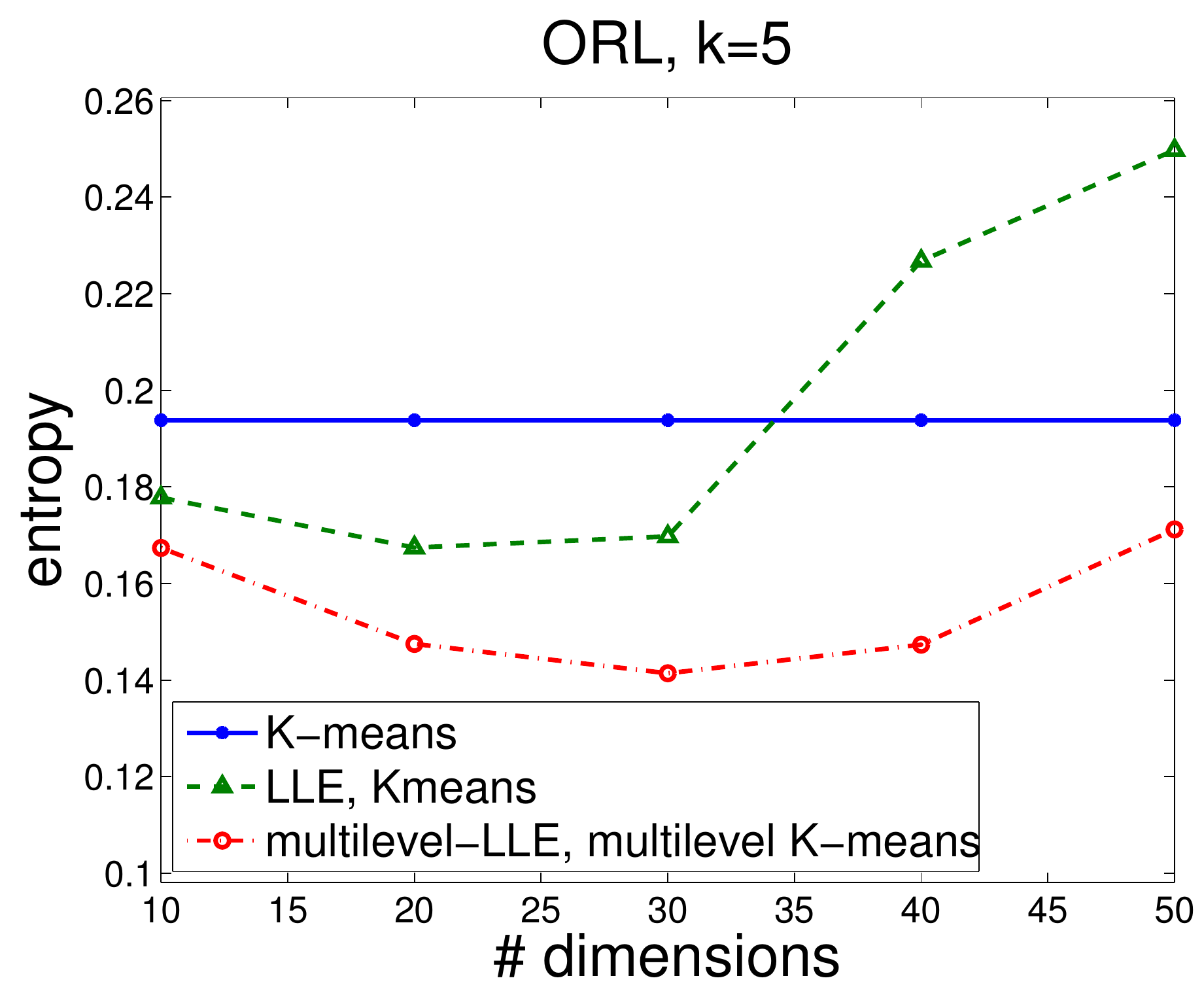} 
\includegraphics[height=0.27\textwidth]{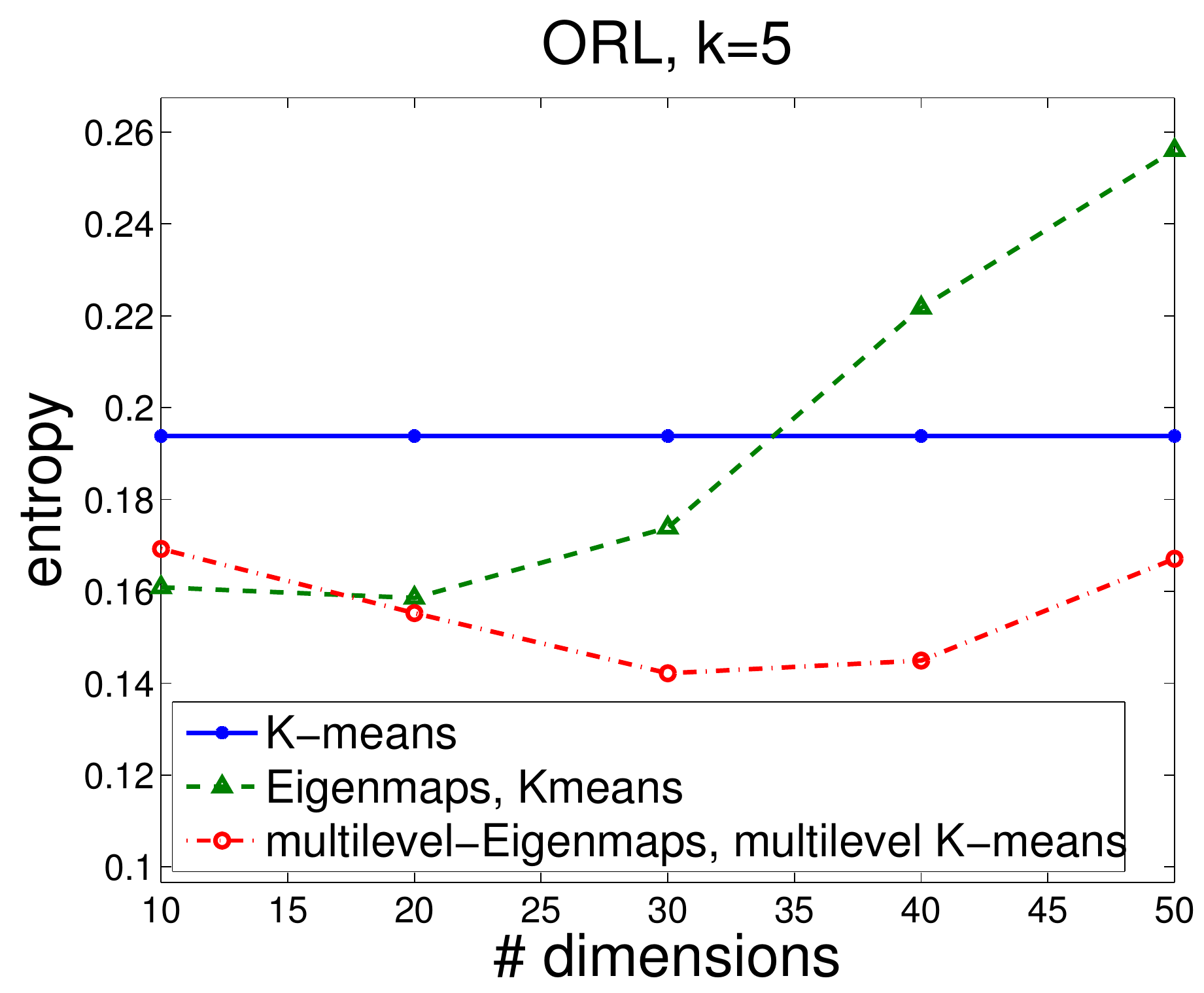} 
\vskip -0.1in
\caption{Purity and entropy values versus  dimensions for three types of clustering for 
ORL dataset.} 
\label{fig:clustering_1}
\end{figure}
\subsubsection{Projective clustering}
The next application we consider is a set of nonlinear  projection based clustering
techniques. We illustrate how the multilevel coarsening methods can be used for data reduction
in this application.
We consider three types of nonlinear projection methods, namely, Isomap~\cite{tenenbaum2000global}, 
Local Linear Embedding (LLE)~\cite{roweis2000nonlinear}
and Laplacian Eigenmaps~\cite{belkin2003laplacian}.
Multilevel algorithm have been used in the clustering application, for example, 
article~\cite{oliveira2005multi} uses a multilevel algorithm, based on  MinMaxCut,
for  document clustering, and Fang et. al.~\cite{fang2010multilevel} applied the 
mutlilevel algorithms for spectral clustering and manifold learning.

Given $n$ data-points, most of the projective clustering  methods start by constructing 
a graph with edges defined based on certain criteria such as
 new distance metrics or  manifolds,
nearest neighbors, points on a same subspace, etc. 
The graph Laplacian corresponding to the graph is considered, and for a given $k$, 
 the top $k$ eigenvectors 
 of a shifted Laplacian matrix,
 whose top eigenvectors correspond
 to the bottom eigenvectors of the original graph,
 are used to  cluster the points. 
 We use the following two evaluation metrics to analyze the  quality of the clusters obtained, namely 
\emph{purity} and \emph{entropy}~\cite{zhao2004empirical} given by:
\begin{eqnarray*}
 \purity = \sum_{i=1}^K\frac{n_i}{n}\purity(i);
 &\purity(i)=\frac{1}{n_i}\max_j(n^j_i),&
\text{and}\\ 
 \entropy = \sum_{i=1}^K\frac{n_i}{n}\entropy(i); &
 \:\entropy(i)=-\sum_{j=1}^K\frac{n^j_i}{n_i}\log_K\frac{n^j_i}{n_i},&
\end{eqnarray*}
where $K$ is the number of clusters, $n^j_i$ is the number of entries of class $j$ in cluster $i$,
and $n_i$ is the number of data in cluster $i$. Here, we  assume that the labels 
indicating the class 
to which data belong
 are available. {The statistical meaning of these metrics can be found in~\cite{zhao2004empirical}, along with empirical and theoretical comparisons of these metrics with other standard metrics.}

In figure~\ref{fig:clustering_1} we present results for three types of projective clustering 
methods, viz., Isomap, LLE and eigenmaps when coarsening was used before dimensionality
reduction. The dataset used is the popular ORL face dataset~\cite{samaria1994parameterisation},
which contains $40$ subjects and $10$ grayscale images each of size $112\times92$ with 
various facial expressions (matrix size is $10304\times 400$).
For the projective methods, we first construct a $k$-nearest neighbor graph with $k=5$,
and use embedding dimensions $p=10,\ldots,50$. Note that even though the data is dense, the 
kNN graph is sparse. The figure presents the purity and entropy values
obtained for the three projective clustering 
methods for these different dimensions $p$ with (circle) and without (triangle) coarsening the graph.
The solid lines indicate the results when kmeans was directly used on the data without dimensionality
reduction. 
We see that the projective methods give improved
 clustering quality in terms of both purity and entropy,
and coarsening further improves their results in many cases by reducing redundancy.
This method was also discussed 
in~\cite{fang2009multilevel} where additional results and illustrations with other
applications can be found.
{We refer to~\cite{fang2009multilevel,fang2010multilevel} for additional details of this application and methods.}

\subsubsection{Genomics - Tagging SNPs}
The third application we consider is that of DNA microarray gene analysis.
The data from microarray experiments is represented
as a matrix $A\in \RR^{m\times n}$, where $A_{ij}$ indicates whether the $j$-th  expression level
exists for gene $i$. Typically, the matrix could have entries $\{-1,0,1\}$ indicating whether the
expression exists ($\pm1$) or not (0) and the sign indicating the order of the sequence, see supplementary material of~\cite{paschou2007intra} for details on 
this encoding. 
Article~\cite{paschou2007intra} used CSSP with a  greedy selection algorithm 
to select a subset of gene expressions or single nucleotide polymorphisms (SNPs)
from a table of SNPs for different populations that capture the spectral information (variations)
of population. The subset of SNPs are called \emph{tagging SNPs} (tSNPs).
Here we show how the coarsening method can be applied in this application to select  columns
(and thus tSNPs) from the table of SNPs, which characterize the extent
to which major patterns of variation of the intrapopulation data
are  captured  by  a  small  number  of  tSNPs.

\begin{table*}[tb!]
\caption{TaggingSNP: Coarsening, Leverage Score sampling and Greedy selection}
\label{tab:tablegen}
{\small
\begin{center}
\begin{tabular}{|l|c|c|c|c|c|}
\hline
Data& Size&$c$ &Coarsen & Lev. Samp. &Greedy \\
\hline
Yaledataset/SORCS3& $1966\times53$ & 14 & 0.0893& 0.1057  & 0.0494\\
Yaledataset/PAH& $1979\times32$ & 9 & 0.1210& 0.2210  & 0.0966\\
Yaledataset/HOXB& $1953\times96$ & 24 & 0.1083& 0.1624 & 0.0595\\
Yaledataset/17q25& $1962\times63$ & 16 & 0.2239& 0.2544 & 0.1595\\
\hline
HapMap/SORCS3& $268\times307$ & 39 & 0.0325& 0.0447 &0.0104\\
HapMap/PAH& $266\times88$ & 22 & 0.0643&0.0777  &  0.0311\\
HapMap/HOXB& $269\times571$ & 72 & 0.0258& 0.0428  &0.0111\\
HapMap/17q25& $265\times370$ & 47 & 0.0821&  0.1190  & 0.0533\\
\hline
\end{tabular}
\end{center}
}
\end{table*}

We use the same two datasets as in~\cite{paschou2007intra}, namely the Yale dataset and the Hapmap datset.
The Yale dataset\footnote{\url{http://alfred.med.yale.edu/}}~\cite{osier2001alfred} contains 
a total of 248 SNPs for around 2000 unrelated individuals from 38
populations from around the world.
 We consider four  genomic regions
(\emph{SORCS3,PAH,HOXB,}  and  \emph{17q25}).
The HapMap project\footnote{\url{https://www.ncbi.nlm.nih.gov/variation/news/NCBI_retiring_HapMap/}}~\cite{gibbs2003international}
(phase I)  released a public database of 1,000,000 SNP typed in different
populations. From this database, we consider the data for the same four regions.
Using the  SNP table, an  encoding matrix  $A$ is formed  with entries
$\{-1, 0, 1\}$, with the same  meaning of the three possible values as
discussed above.
We obtained these encoded matrices, made available online by the authors 
of~\cite{paschou2007intra}, from~\url{http://www.asifj.org/}. 

Table~\ref{tab:tablegen}  lists the  errors  obtained  from the  three
different  methods, namely,  Coarsening, Leverage  Score sampling  and
Greedy      selection~\cite{paschou2007intra}       for      different
populations. 
If  $nnz(X)$  is the  number of  nonzero elements  in a matrix $X$,
the error that is reported is given by  $nnz(\hat A- A)/nnz(A)$,
where $A$ is  the input encoding matrix, 
 $\hat{A}=CC^\dag  A$, is  the projection of  $A$ onto  
the sampled/coarsened $C$.
  The greedy  algorithm
considers  each  column  of  the  matrix  sequentially,  projects  the
remaining columns  onto the considered  column and chooses  the column
that gives least  error as defined above.  The  algorithm then repeats
the procedure to  select the next column and so  on. This algorithm is
very expensive  but it performs  rather well in practice. Observe
that  the coarsening  algorithm  performs better  than leverage  score
sampling  and the  performance is  comparable with that  of the  greedy
algorithm  in some  cases. The coarsening  approach is  less expensive
than leverage score sampling  which in turn is much less expensive  than
the greedy algorithm.

\subsubsection{Multilabel Classification}

The last application we consider  is that of multilabel classification
(MLC). As seen in section~\ref{sec:appl},  the most common approach to
handle large  number of labels in  this problem is to  perform a label
dimension  reduction assuming  a low  rank property  of labels,  i.e.,
assuming  that only a few labels  are important.  In this  section, we
propose   to  reduce   the   label  dimension   based  on   hypergraph
coarsening. Article~\cite{bi2013efficient} presented  a method for MLC
based on  CSSP using leverage score  sampling. The idea is  to replace
sampling by hypergraph coarsening in this method.
    
    \begin{table*}[tb!]
\caption{Multilabel Classification using CSSP (leverage score sampling) and coarsening: 
Average training and test errors and $Precison@k$, $k=$sparsity.}
\label{tab:tabmlc}
\begin{center}
{\small
\begin{tabular}{|l|c|c|c|c|c|c|}
\hline
 Data  & Method & $c$& Train Err & Train P@k & Test Err &Test P@k\\
 \hline

 \hline
  \multirow{2}{12em}{Mediamill,  $d=101, 
  n=10000, nt=2001,
 p=120$.} 
  &Coars	&51&{\bf10.487}	&0.766		&{\bf8.707}		&{\bf0.713}\\
 &CSSP		&51 &10.520	&{\bf0.782}		&12.17		&0.377\\

 \hline
  \multirow{2}{12em}{Bibtex,  $d=159, 
  n=6000, nt=1501,
 p=1836$.} 
  &Coars	&80&{\bf1.440}	&{\bf0.705}		&4.533		&{\bf0.383}\\
 &CSSP		&80 &1.575	&0.618		&{\bf4.293}		&0.380\\
 
 \hline
  \multirow{2}{12em}{Delicious,  $d=983, 
  n=5000, nt=1000,
 p=500$.} 
  &Coars	&246&{\bf50.943}	&0.639		&{\bf74.852}		&0.455\\
 &CSSP		&246 &53.222	&{\bf0.655}		&77.937		&{\bf0.468}\\
 \hline
  \multirow{2}{12em}{Eurlex,  $d=3993, 
  n=5000, nt=1000,
 p=5000$.} 
  &Coars	&500&2.554	&{\bf0.591}		&{\bf73.577}		&0.3485\\
 &CSSP		&500 &{\bf2.246}	&0.504		&81.989		&{\bf0.370}\\
  \hline
\end{tabular}
}
\end{center}
\end{table*}


Table~\ref{tab:tabmlc}  lists  the  results   obtained  for  MLC  when
coarsening  and leverage  score sampling  (CSSP) were  used for  label
reduction  in  the  algorithm of~\cite{bi2013efficient}  on  different
popular  multilabel   datasets.  All   datasets  were   obtained  from
\url{https://manikvarma.github.io/downloads/XC/XMLRepository.html}.
The gist  of the ML-CSSP  algorithm is as  follows: Given data  with a
large   number   of   labels   $Y\in\mathbb{B}^{n\times   d}$,   where
$\mathbb{B}$ is a  binary field with entries $\{0,1\}$,  the
label dimension is reduced 
by  subsampling or coarsening the  label matrix leading to
$c<d$ labels.  The next step is to  train  $c$ binary
classifiers for  these reduced $c$ labels.   For a new data  point, we
can predict whether  the data-point belongs to the  $c$ reduced labels
using  the  $c$  binary  classifiers, by  getting  a  $c$  dimensional
predicted label vector.  We then project the predicted vector onto $d$
dimension  and then  use rounding  to  get the  final $d$  dimensional
predicted vector.
     
All prediction  errors reported (training  and test) are  Hamming loss
errors,  i.e., the  number of  classes for  which the  predicted label
vector differs from the exact label  vector. The second metric used is
$Precison@k$,   which    is   a  common   metric   used    in the  MLC
literature~\cite{ubaru17a}. It measures  the precision of predicting
the first $k$ coordinates $|supp(\hat{y}_{1:k})\cap supp(y)|/k$, where
$supp(x)=\{i|x_i\neq0\}$.   In the  above  results,  we chose  $k=$the
actual sparsity of the predicted  label vector.  This is equivalent to
checking whether or  not the proposed method predicted  all the labels
the  data  belongs  to  correctly.    Other  values  of  $k$  such  as
Precision@$k$ for  $k=1,3,5$ are used,  where one is  checking whether
the top 1,3 or 5 labels respectively are predicted correctly, ignoring
other and false labels. The better  of the two results is highlighted.
In this  application too, we  see that the coarsening  method performs
well, outperforming  the more costly CSSP method in a few cases.

\section{Conclusion}
This paper advocated the use of coarsening techniques for a number of
matrix approximation problems, including the computation of 
a partial SVD, the column  subset selection problem,
and graph sparsification. 
We  illustrated how coarsening methods,
and a combination of sampling and coarsening, can be applied to
solve these  problems and also presented  a few (new) applications  for 
coarsening technique. 

The experiments showed  that the techniques based on coarsening 
perform quite well in  practice, better than
the randomized methods in many cases. 
Coarsening is  also inexpensive  compared to the sampling methods for large sparse matrices.   
While coarsening has traditionally been exploited in 
a completely different context to devise multilevel schemes for 
sparse systems as well as for graph partitioning, it appears that
the same principles offer a tremendous potential for solving 
problems related to data.

{\small
\bibliographystyle{abbrv}
\bibliography{mlsi}
}
\end{document}